\documentclass[reqno,twoside,11pt]{amsart}

\usepackage{amsmath,amsfonts,calrsfs,fullpage,amssymb,color,verbatim,eucal,yfonts,mathrsfs}

\newtheorem{Theorem}{Theorem}[section]
\newtheorem{Definition}[Theorem]{Definition}
\newtheorem{Proposition}[Theorem]{Proposition}
\newtheorem{Lemma}[Theorem]{Lemma}
\newtheorem{Corollary}[Theorem]{Corollary}
\newtheorem{Remark}[Theorem]{Remark}

\newtheorem{Hypothesis}[Theorem]{Hypothesis}
\makeatletter
\@addtoreset{equation}{section}

\makeatother

\def\R{\mathbb R}
\def\N{\mathbb N}

\def\E{\mathbb E}
\def\P{\mathbb P}

\def\eps{\varepsilon}
\def\ds{\displaystyle}

\newcommand{\esssup}{\operatorname{ess\,sup}}
\newcommand{\essinf}{\operatorname{ess\,inf}}

\newcommand{\one}{1\mkern -4mu\mathrm{l}}

\title[Surface integrals for general measures in Hilbert spaces]{\bf Malliavin Calculus for non Gaussian differentiable measures and surface measures in Hilbert spaces}

\author[G. Da Prato]{Giuseppe Da Prato}
\address{Scuola Normale Superiore\\
Piazza dei Cavalieri, 7\\ 
56126 Pisa, Italy}
\email{g.daprato@sns.it}

\author[A. Lunardi]{Alessandra Lunardi}
\address{
Dipartimento di Matematica e Informatica\\
Universit\`a di Parma\\
Parco Area delle Scienze, 53/A\\
43124 Parma, Italy}
\email{alessandra.lunardi@unipr.it}

\author[L. Tubaro]{Luciano Tubaro}
\address{
Dipartimento di Matematica\\
Universit\`a di Trento\\
Via Sommarive 14\\
38123 Povo, Italy}
\email{tubaro@science.unitn.it}



\begin{document}

 \begin{abstract}  
We construct surface measures in a Hilbert space  endowed with a probability measure $\nu$.  The theory fits 
for invariant measures of some stochastic partial differential equations such as  Burgers and reaction--diffusion equations. 
Other examples are weighted Gaussian measures and special product measures $\nu$ of non Gaussian measures; in this case we exhibit a Markov process having $\nu$ as invariant measure. In any case we prove integration by parts formulae on sublevel sets of good functions (including spheres and hyperplanes) that involve  surface integrals. 

 \end{abstract}

\maketitle

 \tableofcontents 
  

\section{Introduction}
Let $X$ be a separable  infinite dimensional Hilbert space with norm $\|\cdot \|$ and  inner product $\langle\cdot,\cdot   \rangle$, endowed with a non degenerate  Borel probability measure   $\nu$.  

 In this paper we define Sobolev spaces with respect to $\nu$, 
we construct surface measures naturally associated to $\nu$, and we describe their main properties. In particular, we aim at  integration by parts formulae for Sobolev functions, that involve traces of Sobolev functions on regular surfaces, and to an infinite dimensional  (non Gaussian) version  of the Divergence Theorem.

The surfaces considered here are level surfaces of a  Borel  function  $g $  that satisfies some   regularity and nondegeneracy assumptions,  which guarantee that such   level surfaces are smooth enough. 

In the case of Gaussian measures this problem has been extensively studied by different approaches. We quote here  \cite{Sk74,Ug79,AiMa88,FePr92,Ma97,Boga,Hi10,AMMP,DaLuTu14}, for an extensive bibliography see the review paper \cite{BogaReview}. 

The approach initiated by Airault and Malliavin in  \cite{AiMa88} for the Wiener measure in the space $X= \{ f\in C([0,1];\R):\; f(0)=0\}$ is naturally extendable to many other settings. It consists in the study of the function
$$F_{\varphi}(r )= \int_{\{x:\,g(x) \leq  r\}} \varphi(x)\nu(dx), \quad r\in \R ,  $$
which is well defined for every  $\varphi\in L^1(X, \nu)$. If $F_{\varphi}$ is differentiable at $r$, its derivative $F_{\varphi}'(r )$ is the candidate to be a surface integral, 
\begin{equation}
\label{scopo}
F_{\varphi}'(r ) = \int_X \varphi\,d\sigma^g_r . 
\end{equation}
It turns out that  $F_{\varphi}$ is differentiable   for good enough functions $\varphi$, and the second step of the construction is to show that 
there exists a measure $\sigma_r^g$ such that  \eqref{scopo} holds. Then, one needs to show that for every $r\in \R$, $\sigma_r^g $ is supported in $g^{-1}(r)$ for a suitable version of $g$, and to clarify the dependence on $g$. The equality  \eqref{scopo} is also a useful tool to prove an infinite dimensional version of the Divergence Theorem (or, of  integration by parts formulae). 
 This approach was  followed e.g. in \cite{Boga,DaLuTu14}, for Gaussian measures in Banach spaces, and in \cite{BogaMal} for general differentiable measures. Notice that if $\varphi \equiv 1$ and $g(x)$ is the distance of $x$ from a given hypersurface $\Sigma$, $F_1'(0)$ is just   the Minkowski content of $\Sigma$. 

A completely different approach is the one by Feyel and de La Pradelle, who constructed an infinite dimensional Hausdorff--Gauss surface measure by approximation with finite dimensional Hausdorff--Gauss surface measures \cite{FePr92}.  It uses in a very important way the structure of Gaussian  measures and it seems to be hardly extendable to non Gaussian settings, especially in the case of non product measures. 

A third approach comes from the general geometric measure theory, that relies on the theory of the BV functions (functions with bounded variation). BV functions for Gaussian measures in Banach spaces were studied e.g. in \cite{F,FH,AMMP}. By definition, a  Borel set $B$ has finite perimeter if its characteristic function is BV; in this case the perimeter measure is defined and its support is contained in the boundary of $B$. For good enough sets $B$, the perimeter measure coincides with the restriction to the boundary of $B$ of the surface measure  of Feyel and de La Pradelle; for a proof see \cite{CeLu14}.  
 
In our general framework we shall follow the first approach, and we are particularly interested in the case where $\nu$ is the invariant measure of some nonlinear
stochastic PDE. In the  case of linear equations, $\nu$ is a Gaussian measure and we refer to our paper \cite{DaLuTu14}. 

Let us describe our procedure. 
As usual, we denote by $C^1_b(X)$ the space of the bounded and continuously Fr\'echet differentiable functions $f:X \mapsto \R$ having  gradient with bounded norm,  by $\nabla f(x)$ the gradient of $f$ at $x$, and by $\partial_zf(x) = \langle \nabla f(x), z\rangle$ the derivative of $f$ at $x$ along any $z\in X$. 

 Our starting assumption is the following. 

\begin{Hypothesis}
\label{h1}
 There exists a linear bounded operator $R\in {\mathcal L}(X)$ such that $R\nabla :$ dom$\,(R\nabla)= C^1_b(X)\mapsto L^p(X, \nu; X)$ is closable in $L^p(X,\nu)$,  for any $p\in (1, +\infty)$. 
 \end{Hypothesis}

Then we denote by $W^{1,p}(X,\nu)$ the domain of the closure  $M_p$ of $R\nabla$   in $L^p(X,\nu)$. $W^{1,p}(X,\nu)$ is a Banach space with the graph norm, 
\begin{equation}
\label{graphnorm}
\|f\|_{W^{1,p}} = \bigg( \int_X|f(x)|^p\nu(dx)\bigg)^{1/p} +  \bigg( \int_X\|M_pf(x)\|^p\nu(dx)\bigg)^{1/p}. 
\end{equation}
So, by definition an element $f\in L^p(X, \nu)$ belongs to $W^{1,p}(X,\nu)$ iff there exists a sequence of $C^1_b$ functions $(f_n)$ such that $\lim_{n\to \infty} f_n =f$ in $ L^p(X, \nu)$ and the sequence $(R\nabla f_n)$ converges in $L^p(X, \nu; X)$, the limit of the latter is just $M_pf$.

Different choices of $R$ give rise to different Sobolev spaces. For instance, if $\nu$ is the  Gaussian measure $N_{0,Q}$ with mean $0$ and covariance $Q$, Hypothesis \ref{h1} is satisfied by $R= Q^{\alpha}$, for every $\alpha \geq 0$. Taking $\alpha =0$ and $R=I$ we obtain the Sobolev spaces studied in  \cite{DPZ}, taking $\alpha = 1/2$ we obtain $W^{1,p}(X,\nu)= {\mathbb D}^{1,p}(X, \nu)$, the usual Sobolev spaces of Malliavin calculus (\cite{Boga,Nualart}).

For general results ensuring that Hypothesis \ref{h1} holds  we quote \cite{AlRo90}.  An easy sufficient condition for $R\nabla $  to be closable in  $L^p(X,\nu)$ is the following one.

\begin{Hypothesis}
\label{h1'}
  For any $p>1$ and  $z\in X$ there exists $C_{p,z}>0$ such that  
\begin{equation}
\label{e1a}
\left|\int_X\langle R\nabla \varphi,z \rangle\,d\nu\right|\le C_{p,z}\,\|\varphi\|_{L^p(X,\nu)} , \quad \varphi\in C^1_b(X). 
\end{equation}
\end{Hypothesis}

In this case,   $\nu$ is Fomin differentiable along $R^*(X)$.  We refer to \cite{Boga2} for a general treatment of differentiable measures.

After the canonical identifications of the dual spaces  $(L^p(X, \nu))'$,  $(L^{p}(X,\nu;X))'$ with $L^{p'}(X,\nu)$, 
 $L^{p'}(X,\nu;X)$ respectively, with $p'=p/(p-1)$  (\cite{DU}), we denote  by $M_p^*: D(M_p^*)\subset 
L^{p'}(X,\nu;X)\to L^{p'}(X,\nu)$  the adjoint of $M_p$.
So, we have
\begin{equation}
\label{e1m}
\int_X \langle M_p\varphi,F  \rangle\,d\nu=\int_X \varphi\,M_p^*(F)\,d\nu,\quad \varphi\in D(M_p),\;F\in D(M_p^*).
\end{equation}
 In the case that  $\nu$ is the Gaussian measure $N_{0,Q}$,   taking $R=Q^{1/2}$, $M_p$ is the Malliavin derivative and $-M_p^*$ is   the Gaussian divergence or   Skorohod integral. See e.g. \cite{Boga,Nualart,Sa05}. 
In any case, 
the operator $-M_p^*$ plays the important role of (generalized) divergence.

Hypothesis \ref{h1'}  is equivalent to the assumption  that  for every 
 $z\in X$   the constant vector field $F_z(x) := z$ belongs 
to $D(M_p^*)$ for every $p>1$.  Indeed, fixed any $p>1$, $F_z\in D(M_p^*)$ iff the function $W^{1,p}(X, \nu)\mapsto \R$, $\varphi\to \int_X \langle M_p\varphi, z\rangle d\nu$ has a linear continuous extension to the whole $L^p(X, \nu)$. Since $C^1_b(X)$ is dense in $W^{1,p}(X, \nu)$, this is equivalent to the existence of $C_{p,z}$ such that \eqref{e1a}  holds, and  in this case   \eqref{e1m}, with  $F=F_z$ and $M_p^*(F_z)=:v_z$, reads as
 \begin{equation}
\label{e2m}
   \int_X  \langle M_p\varphi,z\rangle\, d\nu=\int_X\varphi\, v_z\,d\nu , \quad \varphi \in D(M_p).
\end{equation}
This is a natural generalization of the integration formula that holds for the Gaussian measure $N_{0,Q}$,   in which case taking $R= Q^{1/2}$, \eqref{e1a} holds for every $z\in X$. Moreover $v_z$ is an element of $L^q(X, \nu)$ for every $q\in (1, +\infty)$, it coincides with $  \langle Q^{-1 }z, \cdot \rangle$ if in addition $z\in Q (X)$.

Under Hypothesis \ref{h1'}, formula \eqref{scopo} is a useful tool to prove an integration formula, 
\begin{equation}
\label{parti}
 \int_{ \{g<r  \}} \langle M_p\varphi, z\rangle  \, d\nu = \int_{\{g<r\}}v_z \varphi \, d\nu + \int_{\{g=r\}} \varphi \langle \frac{M_pg }{\|M_pg\|}, z\rangle \,d \rho_r , 
\end{equation}
for all $z\in X$, $\rho_r = \|M_pg\| \sigma^g_r$, and for good enough $\varphi$ and $g$. The  normalized measure $\rho_r$ is particularly meaningful, since it is independent of the choice of $g$ within a large class of functions, being a sort of perimeter measure relevant to the set $\Omega : = g^{-1} (-\infty, r)$ (see Section 5).

 We already mentioned that we need some regularity/nondegeneracy conditions on $g$. Specifically, our assumption on $g$ is
\begin{Hypothesis}
\label{h2}
$g\in W^{1,p}(X,\nu)$ and $ M_pg\,\|M_pg\|^{-2}$ belongs to the domain of the adjoint $M_p^*$, for every $p>1$.
\end{Hypothesis}
So, regularity is meant as Sobolev regularity. The nondegeneracy condition is hidden in the condition that $ M_pg\,\|M_pg\|^{-2}$ belongs to $D(M_p^*)$ for every $p>1$. Indeed, if a vector field $F$ belongs to $D(M_p^*)$, then $\|F\|\in L^{p'}(X, \nu )$. If $g$ satisfies Hypothesis \ref{h2}, taking  $F =  M_pg\,\|M_pg\|^{-2}$, we obtain that $1/\|M_pg\| \in L^{p'}(X, \nu )$, for every $p'>1$. This condition is a generalization of the nondegeneracy condition of  \cite{AiMa88}. 
We recall that if $g$ is  smooth, its level surfaces are smooth near every point $x$ such that $\nabla g(x)\neq 0$. Here what replaces the gradient of $g$ is $M_pg$. $M_pg$ is allowed to vanish at some points, but not too much, otherwise $1/\|M_pg\|$ cannot belong to all $L^{p'}(X, \nu )$ spaces. 

Let us describe the content of the paper. 
 
In Section 2 we define Sobolev spaces and we prove their basic properties, and their properties that are useful for the construction of surface measures.

In Section 3 we
 construct  surface measures under Hypotheses \ref{h1} and \ref{h2}.  

In Section 4 we introduce and discuss the  $p$-capacities, that are used to obtain further properties of the surface measures. In particular, we show that Borel sets with null $p$-capacity for some $p>1$ are negligible with respect to our surface measures. 

Section 5 deals with  a comparison with a geometric measure theory approach, and to the proof of a variational result. Indeed, we show that for every $\varphi\in C^1_b(X)$ with nonnegative values, the integral of $\varphi$ with respect to $\rho_r$ is equal to the maximum of 
$$\int_{\Omega} M_p^*(F\varphi) \, d\nu, $$
where $\Omega = g^{-1} (-\infty, r)$, and $F$ runs among suitably smooth $X$-valued vector fields such that $\|F(x)\|=1$ for $\nu$-a.e $x\in X$. 

Sections 6, 7 and 8   are devoted to examples. In all of them we show that Hypothesis \ref{h1'} holds, and therefore Hypothesis \ref{h1} holds. Moreover, in all of them we prove that the functions $g(x) = \|x\|^2$ and $g(x) = \langle b, x\rangle$, with any $b\in X \setminus \{ 0\}$, satisfy Hypothesis \ref{h2}. 

In Section 6 we consider a weighted Gaussian measure, $\nu(dx) = w(x)\mu(dx)$, where $\mu$ is a nondegenerate centered Gaussian measure. Under suitable conditions on the weight $w$ and on $g$ we show that for every $r\in (\essinf g, \esssup g)$, $\rho_r$ coincides with the restriction of the weighted measure $w(x)\rho(dx)$ to the surface $g^{-1}(r)$, where $\rho$ is the above mentioned Gauss-Hausdorff measure of Feyel and de La Pradelle. Here we consider precise versions of $w$ and $g$, that are elements of Sobolev spaces without a continuous version in general. The results of Section 6 rely on \cite{Simone},  where weighted Gaussian measures in Banach spaces are studied. 

In Section 7 we introduce
an infinite product of   non Gaussian measures on $\R$, which is one of the  simplest generalizations of a Gaussian measure in a separable Hilbert space.  It is an invariant measure of a Markov semigroup $P_t$, described in $\S$ 7.1.  In this toy example we have explicit formulae for all the objects involved: $\nu$, $v_z$, $P_t$.  

 In Section 8  we consider the invariant measures of two particular stochastic PDEs. The first one is a reaction-diffusion equation with a polynomial nonlinearity, and the second one is a Burgers equation. In both cases a unique invariant measure $\nu$ exists, but it is not explicit in general. It is not a product measure, or a Gaussian measure with weight  (except in the case of reaction-diffusion equations, for a particular value of a parameter). However,  Hypothesis \ref{h1'}  is satisfied for every $p>1$ thanks to recent results  (\cite{DaDe16,DaDe15}) that allow  our machinery to work, taking as $R$ a suitable   power of the negative Dirichlet Laplacian. 
 
The verification of Hypothesis \ref{h2}  may be non trivial, since  $\nu$ is not explicit. (In fact, it may be nontrivial even for Gaussian measures, if $g$ is particularly nasty). It is reduced to show that $1/\|M_pg\|$ belongs to $L^p(X, \nu)$ for every $p$, and this is difficult to check, except for hyperplanes in which case $g(x) = \langle b, x\rangle$ for some $b\in X\setminus \{0\}$ and $M_pg$ is constant. 
We show that it holds 
 in the case of spherical surfaces,  when $g(x) = \|x\|^2$. In this  case, the problem is reduced to show that $x\mapsto \|R\nabla g(x)\|^{-1 } = \|2Rx\|^{-1}$ belongs to $L^p(X, \nu)$ for every $p>1$. 
 To show it we need some technical tools, namely we approximate $ \|R\nabla g \|^{-1}$ by a sequence of cylindrical functions $\varphi_n$ belonging to the domain of the infinitesimal generator $L$ of the transition semigroup in $L^2(X, \nu)$. For   functions $\varphi\in D(L)$ we know that $\int_X  L\varphi \,d\nu =0$, and  we use this equality to  estimate the $L^p$ norm of $\varphi_n$ by a constant independent of $n$. 
 
Section 9 contains just some comments and bibliographical remarks.

\section{Notation and preliminaries, Sobolev spaces}

As mentioned in the introduction, we consider  a separable Hilbert space  $X$ with norm $\|\cdot\|$ and scalar product $\langle\cdot, \cdot\rangle$,  endowed with a Borel non degenerate probability measure $\nu$.

We recall that for Fr\'echet differentiable functions $\varphi: X\mapsto \R$ we denote by  $\nabla \varphi(x)$ the  gradient of $\varphi $  at $x$, and by $\partial_z\varphi(x) = \langle \nabla \varphi(x), z\rangle$ its derivative along $z$, for every $z\in X$. 

By $C_b(X)$ (resp. $UC_b(X)$) we mean the space of all real continuous (resp. uniformly continuous) and bounded mappings   $\varphi:  X\to \R$, endowed with the sup norm $\|\cdot \|_{\infty}$. Moreover, $C^1_b(X)$  is  the subspace of $C_b(X)$  of all continuously Fr\'echet  differentiable functions, with bounded (resp. uniformly continuous and bounded) gradient. 

For $p>1$ we set as usual $p' = p/(p-1)$. 

 Throughout the paper we assume that Hypothesis \ref{h1} holds. The spaces $W^{1,p}(X, \nu)$ and the operators $M_p$ are defined in the introduction. Here we collect some of their basic properties. 

\begin{Lemma} 
\label{prodotto-ChainRule}
Let $1<p<\infty$. 
\begin{itemize}
\item[(i)]
If  $\varphi \in W^{1,p}(X,\nu)$, $\psi \in C^1_b(X)$ then  the product $\varphi \psi$ belongs to $W^{1,p}(X,\nu)$ and $M_p(\varphi \psi) = \psi M_{p}\varphi + \varphi  M_{p}\psi$. More generally, if 
 $\varphi \in W^{1,p_1}(X,\nu)$, $\psi \in W^{1,p_2}(X,\nu)$, and  $1/p_1 + 1/p_2 < 1$, then  the product $\varphi \psi$ belongs to $W^{1,p}(X,\nu)$ and $M_p(\varphi \psi) = \psi M_{p_1}\varphi + \varphi  M_{p_2}\psi$, with 
$$\frac{1}{p} = \frac{1}{p_1} + \frac{1}{p_2}. $$
\item[(ii)] Let $h\in C^1_b(\R)$ and $\varphi\in W^{1,p}(X,\nu)$.
Then $h\circ \varphi \in W^{1,p}(X,\nu)$ and we have
\begin{equation}
\label{e1h}
M_p(h\circ \varphi)=h'(\varphi) M_p\varphi.
\end{equation}
\item[(iii)] If $\varphi \in W^{1,p}(X,\nu)$, $\varphi(x) \geq 0$ for $\nu$-a.e. $x\in X$, then $x\mapsto (\varphi(x))^s\in W^{1, p/s}(X, \nu)$, for every $s\in (1, p)$, and 
\begin{equation}
\label{varphi^s}
\|\varphi ^s\|_{W^{1, p/s}(X, \nu)} \leq \|\varphi \|_{L^p(X, \nu)}^s + s \|M_p\varphi\|_{L^p(X, \nu; X)}  \|\varphi \|_{L^p(X, \nu)}^{s-1}. 
\end{equation}

\item[(iv)] 
For $1<p<\infty$,  $W^{1,p}(X, \nu)$ is reflexive. 

\item[(v)] 
If $p\in (1, +\infty) $ and  $f_n\in W^{1,p}(X, \nu)$, $n\in \N$, are such that  $f_n\to f$ in $L^p(X, \nu)$ and  $M_pf_n$ is bounded in $L^p(X, \nu; X)$, then
$f\in W^{1,p}(X, \nu)$. 

\item[(vi)] $W^{1,p}(X, \nu)\subset W^{1,q}(X, \nu)$ and $M_p\varphi = M_q\varphi$ for every $\varphi\in W^{1,p}(X, \nu)$,  for $1<q<p$. 
\end{itemize}
\end{Lemma}
\begin{proof} The proof of statement (i) follows by  approaching  $\varphi \psi$ by $  \varphi_n \psi_n$, for any couple of  sequences $(\varphi_n)$, $(\psi_n)\subset C^1_b(X)$ that approach $\varphi$, $\psi$  in $W^{1,p_1}(X,\nu)$, $W^{1,p_2}(X,\nu)$, respectively. Of course if $\psi \in C^1_b(X)$ we take $\psi_n = \psi$ for every $n$. 

\vspace{3mm}

Concerning  statement (ii) we have just to approach $h\circ \varphi$ by $h\circ \varphi_n$, for any sequence $(\varphi_n)\subset C^1_b(X)$ that approaches $\varphi$ in $W^{1,p}(X,\nu)$. 

\vspace{3mm}

Let us prove  (iii). For every sequence $(\varphi_n) \subset C^1_b(X)$ such  that  $\lim_{n\to \infty} \varphi_n = \varphi$ in $W^{1,p}(X, \nu)$, the sequence $\psi_n(x):= \sqrt{\varphi_n(x)^2 + 1/n}$ has a subsequence $(\psi_{n_k})$ such that $\psi_{n_k}(x)\to \varphi(x)$, $R\nabla \psi_{n_k}(x)\to M_p\varphi(x)$ for $\nu$-a.e. $x\in X$, and it is easily seen that   $\psi_{n_k}^s\to  \varphi ^s$   in $L^{p/s}(X, \nu)$ and $M_p\psi_{n_k}^s = s \psi_{n_k}^{s-1} M_p \psi_{n_k} = s  (\varphi_{n_k}^2 + 1/n)^{s-3/2}\varphi_{n_k} M_p\varphi_{n_k}$ converges to 
$s \varphi^{s-1}M_p\varphi$ in $L^{p/s}(X, \nu;X)$. Therefore, $\varphi^s \in W^{1, p/s}(X, \nu)$ and  the H\"older inequality yields  estimate \eqref{varphi^s}. 
 
\vspace{3mm}

 Let us prove statement (iv). 
The mapping $u\mapsto Tu : = (u, M_pu)$ is an isometry from $W^{1,p} (X, \nu)$ to the product space $E:= L^p(X, \nu) \times L^p(X, \nu; X)$, which implies that  the range of $T$ is closed in $E$. Now, $L^p(X, \nu)$ and $L^p(X, \nu; X)$ are reflexive (for the latter statement, see e.g. \cite[Ch. IV]{DU})  so that $E$ is reflexive, and $T(W^{1,p} (X, \nu))$ is reflexive too. Being isometric to a reflexive space,  $W^{1,p} (X, \nu)$ is reflexive. 

\vspace{3mm}

Statement (v) is a consequence of (iv). Since  $(f_n)$ is bounded  in $W^{1,p} (X, \nu)$ which is reflexive, there exists a subsequence that weakly converges to an element of $W^{1,p} (X, \nu)$. Since  $f_n\to f$ in $L^p(X, \nu)$, the weak limit is  $f$. Therefore,  $f\in W^{1,p} (X, \nu)$. 

\vspace{3mm}

Statement (vi) is an immediate consequence of the definition. 
\end{proof}
 
We shall use the following extension of Lemma  \ref{prodotto-ChainRule}(ii) to compositions with piecewise linear functions. 

\begin{Lemma} 
\label{ChainRuleLipschitz}
Let $\alpha <\beta \in \R  $,  and set 
\begin{equation}
\label{defh}
h( r)=\int_{-\infty}^r {\one}_{[\alpha,\beta]}(s)ds =\left\{ \begin{array}{ll} 0 & {\rm if}\;\;r\leq \alpha, \\
 r-\alpha & {\rm if}\; \;\alpha \leq r \leq \beta, \\
 \beta-\alpha & {\rm if}\; \;r\geq \beta .
 \end{array}\right.
 \end{equation} 
Then   $h\circ \varphi \in W^{1,p}(X,\nu)$ for every 
$\varphi\in W^{1,p}(X,\nu)$,  and we have
\begin{equation}
\label{Mph}
M_p(h\circ \varphi)=\one_{[\alpha, \beta]}(\varphi) M_p\varphi.
\end{equation}
\end{Lemma}
\begin{proof} 
 We approach $h$ by a sequence of $C^1_b$ functions, choosing a sequence of smooth compactly supported functions  $\theta_n:\R\mapsto \R$ such that $\theta_n(\xi)\to 
 {\one}_{[\alpha,\beta]}(\xi)$ for every $\xi\in \R$,   $0\leq \theta_n(\xi) \leq 1 $ for every $\xi\in \R$, and setting 
 $$   h_n( r)=\int_{-\infty}^r \theta_n(s)ds, \quad r\in \R. $$
  Since $h_n\in C^1_b(\R)$,  by Lemma \ref{prodotto-ChainRule}(ii) $h_n\circ \varphi  \in W^{1,p}(X,\nu)$, and 
 $$
 M_p(h_n\circ \varphi )=(h_n'\circ \varphi ) M_p\varphi.
 $$
By the Dominated Convergence Theorem, $h_n \circ \varphi$ converges to $h\circ \varphi $ in $L^{p}(X,\nu)$. Moreover, $M_p(h_n\circ \varphi )$ converges pointwise to $\one_{[\alpha, \beta]}(\varphi) M_p\varphi$. Since $\| M_p(h_n\circ \varphi )(x)\|\leq \|\theta_n\|_{\infty}\| M_p\varphi(x)\| \leq \| M_p\varphi(x)\|$, still by the  Dominated Convergence Theorem $ M_p(h_n\circ \varphi )$ converges to $\one_{[\alpha, \beta]}(\varphi) M_p\varphi$ in $L^{p}(X,\nu; X)$, and the statement follows. 
\end{proof}

 \begin{Corollary}
\label{partepositiva}
For every $\varphi\in W^{1,p}(X,\nu)$, the positive  part $\varphi  _+$ of $\varphi $, the negative part $\varphi  _-$ of $\varphi $,  and $|\varphi|$ belong to $W^{1,p}(X,\nu)$, and we have
\begin{equation}
\label{Mpvarie}
M_p(\varphi  _+)=\one_{\varphi^{-1}(0, +\infty)}  M_p\varphi, \quad M_p(\varphi  _-)=-\one_{\varphi^{-1}(-\infty, 0)}  M_p\varphi, \quad M_p(|\varphi  |) = {\rm sign}\; \varphi  \;  M_p\varphi. 
\end{equation}
Moreover, $  M_p\varphi$ vanishes $\nu$-a.e. in the level set $\varphi^{-1}(c)$, for each $c\in \R$. 
\end{Corollary} 
\begin{proof} The proof of the first statement is just a minor modification of the proof of Lemma \ref{ChainRuleLipschitz}; it is sufficient to take $\alpha =0$, $\beta = +\infty$ and approaching functions $\theta_n $ of $ {\one}_{[0,+\infty)}$ that vanish on some left half-line. The other statements are consequences of the first one. \end{proof}

We remark that taking $\alpha =0$, $\beta =1$ and $p=2$ in Lemma \ref{ChainRuleLipschitz}, we obtain that for every $\varphi\in W^{1,2}(X,\nu)$, the function $\varphi  _+ \wedge 1$ belongs to $W^{1,2}(X,\nu)$, and $\| \varphi  _+ \wedge 1\|_{W^{1,2}(X,\nu)} \leq \|\varphi \|_{W^{1,2}(X,\nu)}$. Namely, the quadratic form 
$${\mathcal E}(\varphi, \psi) := \int_X (\varphi \psi + \langle M_2\varphi, M_2\psi\rangle )d\nu , \quad \varphi, \; \psi \in W^{1,2}(X,\nu), $$
is a Dirichlet form. 

In the next lemma we exhibit a class of regular functions that belong to the Sobolev spaces. 

 \begin{Lemma}
\label{C^1crescita}
Let $\varphi\in C^1(X )$ be such that $\| \nabla \varphi\|$ is bounded in $\varphi^{-1}(-r, r)$ for every $r>0$, and 
$$\int_X (|\varphi|^p + \|R \nabla \varphi\|^p)d\nu < \infty. $$
Then $\varphi\in W^{1,p} (X, \nu)$ for every $p\in(1, +\infty)$, and $M_p \varphi= R \nabla \varphi$. 
\end{Lemma}
\begin{proof}
We approach $\varphi$ by regularized truncations, introducing $\theta\in C^{1}_b(\R)$ such that $\theta (\xi )= \xi $ for $|\xi|\leq 1$ and $\theta = $ constant for $\xi \geq 2$ and for $\xi\leq -2$. The functions $\varphi_n(x) := n \theta(\varphi(x)/n)$ belong to $C^{1}_b(X)$, they approach $\varphi$ pointwise and in $L^p(X, \nu)$ by the Dominated Convergence Theorem. Moreover, $R\nabla \varphi_n(x) = \theta'(\varphi(x)/n)R\nabla \varphi(x)$, which coincides with $R\nabla \varphi(x)$ if $|\varphi(x)| \leq n$ and vanishes if $|\varphi(x)| \geq 2n$. Still by the Dominated Convergence Theorem, $R\nabla \varphi_n$ converges to $R\nabla \varphi$ in $L^p(X, \nu; X)$. 

Notice that the assumption that $\| \nabla \varphi\|$ is bounded in $\varphi^{-1}(-r, r)$ for every $r>0$ guarantees that $\| \nabla \varphi_n\|$ is bounded in $X$, so that $\varphi_n\in C^1_b(X)$, for every $n\in \N$. 
\end{proof}

Some properties of the operators $M^*_p$ are in the next lemma.

 \begin{Lemma}
 \label{Le:divergenza} Let $1<p<\infty$. 
\label{h3h}
\begin{itemize}
\item[(i)] For any $F\in D(M_p^*)$ and any $\varphi\in C^1_b(X)$, the product  $\varphi F$ belongs to $D(M_p^*)$ and
\begin{equation}
\label{e1d}
M_p^*(\varphi F)=\varphi  M_p^*(F)-\langle M_p\varphi, F\rangle.
\end{equation}
More generally, for any $F\in D(M_p^*)$ and any $\varphi\in W^{1,q}(X, \nu)$ with $q>p$, the product $\varphi F$ belongs to $D(M_s^*)$ with $s= pq/(q-p)$ and
\eqref{e1d} holds with $s$ replacing $p$. 

\item[(ii)]  For any $F\in D(M_p^*)$, 
\begin{equation}
\label{e1f}
\int_X M_p^*F\, d\nu = 0. 
\end{equation}
\end{itemize}
\end{Lemma}
\begin{proof}
Let $\psi\in  C^1_b(X)$. From the identity $M_p(\varphi \psi) = \varphi M_p\psi + \psi M_p\varphi$ we obtain
\begin{equation}
\label{divprod}
\int_X \langle M_p\psi , \varphi F\rangle d\nu = \int_X \langle M_p(\varphi \psi ) - \psi M_p\varphi, F\rangle d\nu
=  \int_X \psi (\varphi M_p^*F - \langle  M_p\varphi, F\rangle)d\nu , 
\end{equation}
and  the first part of statement (i) follows from the definition of $M_p^*$. The argument is similar if  $\varphi \in W^{1,q}(X, \nu)$; in this case $\varphi \psi \in W^{1,q} (X, \nu) \subset W^{1,p} (X, \nu)$ 
since $p<q$, and we have $M_p(\varphi \psi) = \varphi M_p\psi + \psi M_q\varphi$, while $M_p\psi = M_s\psi$. Formula \eqref{divprod} reads as
$$\int_X \langle M_s\psi , \varphi F\rangle d\nu =  \int_X \psi \,g \,d\nu, $$
where now $g : = \varphi M_p^*F - \langle  M_q\varphi, F\rangle \in L^{s'}(X, \nu)$. 

Since $1\in W^{1,p}(X, \nu)$ and $M_p1=0$, statement (ii) follows from the definition of $M_p^*$.
\end{proof}

\begin{Lemma}
\label{Rem:divergenza}
Let  Hypothesis \ref{h1'} hold, and let $q>1$. Then for every $z\in X$ and $f\in W^{1, q}(X, \nu)$, the vector field 
$$F(x) := f(x)z, \quad x\in X, $$
belongs to $D(M_p^*)$ for every $p>q'$, and
\begin{equation}
\label{formulaM^*_p}
M^*_pF(x) =  -  \langle M_p f(x),  z\rangle  +   v_{z}(x) f(x), \quad x\in X, 
\end{equation}
where $v_z$ is the function in formula \eqref{e2m}. 
\end{Lemma}. 
\begin{proof} 
For every  $\varphi \in C^1_b(X)$ we have
\begin{equation}
\label{formuladiservizio}
\int_X \langle R\nabla \varphi, F\rangle \, d\nu =  \int_X   \langle f M\varphi, z\rangle \, d\nu
= \int_X   (\langle M_p( f \varphi) - \varphi M_pf, z\rangle )\, d\nu = 
\int_X  ( f v_{z}  - \langle M_pf, z\rangle  )\varphi \, d\nu 
\end{equation}
%
%
%
  by Lemma \ref{prodotto-ChainRule}(i) and formula \eqref{e2m}. Since $v_{z}\in L^s(X, \nu)$ for every $s\in (1, +\infty)$, the function  $ ( fv_{z}  - \langle M_pf, z\rangle  )$ belongs to 
$L^s(X, \nu)$ for every $s<q$. 

Approaching every $\varphi \in W^{1,p}(X, \nu)$ by a sequence $(\varphi_n)$ of $C^1_b$ functions, the left-hand side of \eqref{formuladiservizio}  converges to $ \int_X \langle M_p \varphi, F\rangle \, d\nu $. Since $q>p'$, there exists $s\in (1,q)$ such that $s>p'$. So, also the right-hand side converges, and we get 
$$ \int_X \langle M_p \varphi, F\rangle \, d\nu = 
\int_X   ( fv_{z}  - \langle M_qf, z\rangle  )\varphi \, d\nu $$
 for every $\varphi \in W^{1,p}(X, \nu)$.  \eqref{formulaM^*_p} follows from the definition of $M^*_p$. 
 \end{proof}

  
\section{Construction of  surface measures}


We recall that Hypothesis \ref{h1} holds throughout the paper. Moreover, 
from now on,   $g:X\mapsto \R$ is a  Borel  function   that satisfies  Hypothesis  \ref{h2}.  
 
The elements of $W^{1,p}(X,\nu)$  are equivalence classes of functions. If $g$ is a given   function, by $g\in W^{1,p}(X,\nu)$ we mean as usual that $g$ is a fixed  version of an element of $W^{1,p}(X,\nu)$. The results of this section are independent  of the particular chosen version $g$. Instead, in the next section the choice of the version will be important.

We recall  that $W^{1,p}(X, \nu) \subset  W^{1,q}(X, \nu)$ and $M_p \varphi = M_q \varphi$ for every $\varphi \in W^{1,p}(X, \nu)$  if $p>q$ (Lemma \ref{prodotto-ChainRule}(vi)). Therefore, $D(M_p^*) \supset D(M_q^*)$ and 
$M_p^*$ and $M_q^*$ agree on  $D(M_q^*)$ if $p>q$. To simplify notation we shall write $ M$ instead of $M_p$ and $M^*$ instead of $M_p^*$ on $\cap_{p>1}W^{1,p}(X, \nu)$ and on $\cap_{p>1}D(M_p^*) $, respectively. Moreover we 
set
\begin{equation}
\label{psi}
\Psi: = \frac{  Mg}{\|Mg\|^2}. 
\end{equation}

We start our analysis introducing  the function 
\begin{equation}
\label{e14}
F_{\varphi}(r ):=\int_{\{ g\le r\}}\varphi(x)\nu(dx),  \quad r\in \R,\quad \varphi\in L^1(X,\nu).
\end{equation}
We recall that the image measure $(\varphi\nu)\circ g^{-1} $ is defined on the Borel sets $B\subset \R$ by 
$$(\varphi\nu)\circ g^{-1} (B) = \int_{g^{-1}(B)}\varphi(x)\nu(dx). $$
So, $F_{\varphi}(r) = (\varphi\nu)\circ g^{-1}((-\infty, r])$. It is easy to see that $F_{\varphi}$ is continuously differentiable if and only if $(\varphi \nu)\circ g^{-1}$ is absolutely continuous with respect to the Lebesgue measure $\lambda$, with continuous density $q_{\varphi}$. In this case we have
$$F_{\varphi}'(r) = q_{\varphi}(r), \quad r\in \R . $$

So, our next step  is to show   that    $(\varphi\nu)\circ g^{-1} \ll\lambda$, for all $\varphi$ belonging either to $UC_b(X)$ or to $ W^{1,p}(X,\nu)$ for some $p>1$.   Also, we shall show that
the density 
\begin{equation}
\label{e15}
 \frac{d(\varphi\nu)\circ g^{-1}}{d\lambda}\,(r)=:q_\varphi(r),  
\end{equation}
is H\"older continuous if $\varphi\in W^{1,p}(X, \nu)$ for some $p>1$.

It will follow easily that for any $r\in\R$ the mapping $\varphi\mapsto F_{\varphi}'( r)$  is a linear positive functional on $UC_b(X)$, and by results of general measure theory it is indeed the integral of $\varphi$ with respect to a Borel measure. We shall see that such a measure is  concentrated on the surface $\{g=r\}$ if $g$ is continuous, on the surface $\{g^*=r\}$ if $g$ is not continuous, where $g^*$ is a suitable version of $g$. 

The next  lemma is the starting point of most sublevel sets approach to surface measures. Its proof is an abstract version of a well known procedure, see e.g. \cite[First Edition, Prop. 2.1.1]{Nualart}.

\begin{Lemma}
\label{l4}
Assume that Hypotheses \ref{h1} and \ref{h2} are fulfilled.  Then  for any $p>1$ and $\varphi\in W^{1,p}(X, \nu)$,  the measure $ (\varphi\nu)\circ g^{-1}$   is absolutely continuous with respect to  the Lebesgue measure $\lambda$.
Its density 
$$\frac{d[(\varphi\nu)\circ g^{-1}]}{d\lambda}(r)=:q_{\varphi}(r), \quad r\in \R. $$  
 is given by 
\begin{equation}
\label{e17}
q_{\varphi}(r)= \int_{\{g<r\}}\bigg(\langle M_p\varphi, \frac{  Mg}{\|Mg\|^2} \rangle -\varphi M^*_p\bigg(\frac{  Mg}{\|Mg\|^2} \bigg) \bigg)\,d\nu . 
\end{equation}
and it is bounded and  $\theta$-H\"older continuous in $\R$ for every $\theta <1-1/p$. There exists $K_p>0$, independent of $\varphi$, such that   
\begin{equation}
\label{e18}
 | q_{\varphi}(r)| \leq K_p\|\varphi\|_{W^{1,p}(X, \nu)}, \quad r\in \R. 
 \end{equation}
\end{Lemma} 
 \begin{proof}
 Fix any interval $[\alpha,\beta]\subset \R$  and consider the function $h$ defined in \eqref{defh}.  
By Lemma \ref{ChainRuleLipschitz}, $h \circ g \in W^{1,p}(X,\nu)$ for every $p>1$, and 
 $$
 Mn(h\circ g)= \one_{[\alpha, \beta]}(g) Mg.
 $$
   Therefore, 
 $$
\one_{[\alpha, \beta]} \circ g   =\frac{\langle M(h\circ g),Mg \rangle}{\|Mg\|^2} = \langle M(h\circ g),\Psi\rangle ,  $$
where $\Psi$ is defined in \eqref{psi} and belongs to $D(M_p^*)$ for every $p>1$ by Hypothesis \ref{h2}. 
 Let $\varphi\in C^1_b(X)$. Then $ \varphi\Psi\in D(M_p^*)$ for every $p>1$. Multiplying both sides by $\varphi$ and integrating  yields
  $$\int_X \one_{[\alpha,\beta]}(g(x)) \varphi(x) \nu(dx)   =  \int_X (h\circ g)\, M_p^*( \varphi\Psi )\,d\nu.$$
  On the other hand, by Lemma \ref{h3h}(i), $M^*_p( \varphi \Psi ) = M^*(\Psi) \varphi -  \langle M_p\varphi, \Psi \rangle$, and therefore
\begin{equation}
\label{1}
\int_X  \one_{[\alpha,\beta]}(g(x)) \varphi(x) \nu(dx)   =   \int_X (h\circ g)\, (M^*(\Psi) \varphi - \langle M_p\varphi, \Psi \rangle )\,d\nu .
\end{equation}
Approaching any $\varphi\in W^{1,p}(X, \nu)$ by a sequence of $C^1_b$ functions, we see that formula \eqref{1} holds for every $\varphi \in W^{1,p}(X, \nu)$. The right hand side may be rewritten as
$$  \int_X \int_{\R} \one_{(-\infty, g(x)]}(r )\one_{[\alpha, \beta]}( r)dr  \,(M^*(\Psi) \varphi - \langle M_p\varphi, \Psi \rangle )\, d\nu,  $$
so that by the Fubini Theorem,  
$$
(\varphi\nu)(\alpha\leq g\leq \beta)=
\int_\alpha^\beta dr\int_{\{g\geq r\}} (M^*(\Psi) \varphi - \langle M_p\varphi, \Psi \rangle ) \,d\nu .
$$
 Therefore   $(\varphi\nu)\circ g^{-1}$ has  density $q_{\varphi} $ given by
$$
q_{\varphi}(r)=\int_{\{ g\geq  r\}}(M^*(\Psi) \varphi - \langle M_p\varphi, \Psi \rangle ) \,d\nu 
= - \int_{\{g < r\}}( \langle M_p\varphi, \Psi \rangle -M^*(\Psi) \varphi  ) \,d\nu, 
$$
where the last equality follows from Lemma \ref{Le:divergenza}(ii). 
  Since $\Psi\in L^q(X,\nu;X)$ and $M^*\Psi \in L^q(X, \nu)$ for every $q>1$, the function $\langle M_p\varphi, \Psi \rangle -M^*(\Psi)\varphi $ belongs to $L^s(X, \nu)$ for every $s\in  [1, p)$, and there is $C_{p,s}>0$ such that
$$\|
\langle M_p\varphi, \Psi \rangle -M^*(\Psi)\varphi\|_{L^s(X, \nu)} \leq C_{p,s} \|\varphi\|_{W^{1,p}(X, \nu)}. $$
Taking $s=1$, estimate \eqref{e18} is immediate.

  Let us prove that $q_{\varphi}$ is H\"older continuous.  For 
$r_2>r_1$ and for every $s\in (1,p)$  we have
$$|q_{\varphi}(r_2 ) -q_{\varphi} (r_1)| = \bigg|\int_{\{r_1<g\leq r_2\}} (\langle M_p\varphi, \Psi \rangle -M^*(\Psi) \varphi )d\nu \bigg| $$
$$\leq  \|\langle M_p\varphi, \Psi \rangle -M^*(\Psi) \varphi \|_{L^{s}(X, \nu)} \bigg(\int_{r_1}^{r_2}q_1(r)dr\bigg)^{1/s'} 
\leq  C_{p,s} \|\varphi\|_{W^{1,p}(X, \nu)} 
(\|q_1\|_{\infty} (r_2-r_1))^{1-1/s}. $$
Therefore,  $q_{\varphi}$ is H\"older continuous with any exponent less than $1-1/p$.  
 \end{proof}

 Taking in particular $\varphi \equiv 1$, we obtain that $\nu(g^{-1}(r_0)) =
 \int_{r_0}^{r_0}d\nu  =0$ for every $r_0\in \R$. Therefore,  all the level surfaces of $g$ are $\nu$-negligible.  In particular,  
 $$F_{\varphi}(r) = \int_{\{ g  \leq  r\}}  \varphi  \,d\nu = \int_{\{ g < r\}}  \varphi  \,d\nu, \quad \varphi \in L^1(X, \nu). $$
 Moreover, by Lemma \ref{Le:divergenza}, for every $\varphi \in W^{1,p}(X, \nu)$ the product $\varphi  \, Mg/\|Mg\|^2$ belongs to $D(M_s^*)$ for every $s >p'$, and we have 
\begin{equation}
\label{qvarphi}
q_{\varphi}(r)= -\int_{\{ g < r\}} M_s^*\bigg(\varphi \frac{  Mg}{\|Mg\|^2}\bigg)  d\nu . 
\end{equation}
 
 Let us now consider  bounded and uniformly continuous functions $\varphi$. The proof of the next proposition is taken from the paper \cite{DaLuTu14} that deals with Gaussian measures. In the case of general measures there are not  substantial modifications, and the proof is added here just for completeness.
 
  \begin{Proposition}
 \label{p2n}
  For any $\varphi\in UC_b(X)$, $F_{\varphi}$ is continuously differentiable.
 \end{Proposition}
\begin{proof}
  First, let $\varphi\in C^1_b(X)$.
By the Disintegration Theorem, see e.g. \cite[Theorem A1]{DaLuTu14}, we can write
\begin{equation}
\label{disintegrazione}   F_{\varphi}( r)=\int_{-\infty}^r\left(\int_X  \varphi\,dm_s\right)q_1(s)ds 
\end{equation}
where  $m_r $ is a probability measure on ${\mathcal B}(X)$, with support in $\{g=r\}$  for $\lambda$-a.e. $r\in \R$. Here, $\lambda$ is the Lebesgue measure.

    Then  there is a Borel set  $I_\varphi\subset \R$     such that $\lambda(I_\varphi)=0$ and $
   F_{\varphi}$ is differentiable on $\R\setminus I_\varphi$, with 
\begin{equation}
\label{e4n}
 F_{\varphi}' ( r)=  q_1(r) \int_{X}\varphi(x)m_r(dx),\quad\forall\;r\notin I_\varphi. 
\end{equation}
In particular,  
\begin{equation}
\label{e5n}
 F_{1}' ( r)= q_1( r),\quad r\notin I_1. 
\end{equation}
 By \eqref{e4n}  we have
$$
| F_{\varphi}' ( r)|\le \|\varphi\|_\infty q_1 (r ) ,\quad r\notin I_\varphi.
$$
 Taking into account \eqref{e5n}, yields
\begin{equation}
\label{e6n}
 | F_{\varphi}' ( r)|\le | F_{1}' ( r)|\;\|\varphi\|_\infty\quad  r\notin (I_\varphi\cup I_1). 
\end{equation}
Since both $ F_{\varphi}'  $ and $ F_{1}' $
are  continuous,  we have
\begin{equation}
\label{e7n}
 | F_{\varphi}' ( r)|\le | F_{1}' ( r)|\;\|\varphi\|_\infty,\quad  r\in \R,\;\varphi\in C^1_b(X).
\end{equation}

 Now let   $\varphi\in UC_b(X)$ and let $(\varphi_n)$ be a sequence in $C_b^1(X)$  convergent to $\varphi$ in $C_b(X)$ (e.g., \cite{LaLi86}). Then   \eqref{e7n} yields
\begin{equation}
\label{e27}
  |  F_{\varphi_m}' ( r)-  F_{\varphi_n}' ( r) |\le | F_{1}' ( r)|\|\varphi_m-\varphi_n\|_{\infty},\quad  r\in\R.
\end{equation}
Therefore  $(  F_{\varphi_n}' ( r))$ is a Cauchy  sequence in $C_b(X)$,  and the conclusion follows. 
 \end{proof}

The main result of this section is the following.

\begin{Theorem}
\label{costruzione}
Let Hypotheses \ref{h1} and  \ref{h2} hold. Then the function $F_{\varphi}$ is differentiable for every $\varphi\in C_b(X)$. For every $r\in \R$ there exists a  Borel measure $\sigma_r^g$ on $X$ such that 
\begin{equation}
\label{e15d}
F_\varphi '( r)  =\int_{X} \varphi(x)\,\sigma_r^g(dx),\quad \varphi\in C_b(X). 
\end{equation}
In particular, for $\varphi \equiv 1$ we obtain $\sigma_r^g(X)=F_1'(r ) = q_1(r)$.   Therefore, $\sigma_r^g$ is nontrivial  iff $F_1'(r ) > 0$.  
\end{Theorem} 
 \begin{proof} Fix $r\in \R$ and $\varphi \in C_b(X)$. To show that $F_{\varphi}$ is differentiable at $r$, we shall show that for every vanishing sequence $(\eps_n)$ of nonzero numbers the incremental ratio $(F_{\varphi} (r+\eps_n) - F_{\varphi} (r))/\eps_n$ converges to a real limit independent of the sequence, as $n\to \infty$. 

Consider the measures $m_n$ defined by 
 $$m_n = \left\{ \begin{array}{ll}\ds \frac{1}{\eps_n} \one_{g^{-1}(r, r+ \eps_n)} \nu, &{\rm if}\; \eps_n>0, 
 \\
 \\
- \frac{1}{\eps_n} \one_{g^{-1}( r+ \eps_n, r)} \nu, &{\rm if}\; \eps_n<0. 
\end{array}\right. $$
Then $(m_n)$ is a sequence of nonnegative finite Borel (and since $X$ is separable, Radon) measures, and  we have 
$$\int_X \varphi \, dm_n = \frac{1}{\eps_n}(F_{\varphi}(r+\eps_n) - F_{\varphi}(r)). $$

In particular, if $\varphi $ is  Lipschitz continuous  and  bounded by  Proposition \ref{p2n} $F_{\varphi}$ is differentiable, and therefore 
\begin{equation}
\label{m}
\lim_{n\to \infty} \int_X \varphi \, dm_n = F_{\varphi}'(r) = q_{\varphi}(r). 
\end{equation}
So,  the sequence $\int_X \varphi \, dm_n $ converges to $q_{\varphi}(r)$. By a corollary of the Prokhorov Theorem (e.g. \cite[Cor. 8.6.3]{BogaII}), if a sequence  of nonnegative Radon measures $(m_n)$ is such that $\int_X \varphi \, dm_n $ converges in $\R$ for every Lipschitz continuous and bounded $\varphi$,   there exists a limiting Borel measure  such that $(m_n)$ converges weakly to it. The weak limit is independent of the chosen vanishing sequence,  
because for every Lipschitz continuous and bounded $\varphi$ equality \eqref{m} holds, so that denoting by $m$ the weak limit obtained through a sequence  $(\eps_n)$ and by $\widetilde{m}$ the weak limit obtained through another sequence 
 $(\widetilde{\eps_n})$,  we have $\int_X \varphi \, dm = \int_X \varphi \, d\widetilde{m}$ for every Lipschitz continuous and bounded $\varphi$, and this implies that 
 $m= \widetilde{m}$. So, there exists a Borel measure, that we denote by $\sigma^g_r$, such that  for every vanishing sequence $(\eps_n)$ of nonzero numbers, and for every $\varphi\in C_b(X)$ we have
$$\lim_{n\to \infty} \frac{1}{\eps_n}(F_{\varphi}(r+\eps_n) - F_{\varphi}(r)) = \int_X \varphi \, d \sigma^g_r. $$
This means that for every $\varphi\in C_b(X)$ the function $F_{\varphi}$ is differentiable at $r$, and \eqref{e15d} holds. 
\end{proof}

 \begin{Remark}
{\em From the proof of Theorem \ref{costruzione} it follows easily that 
  if $g$ is continuous then $\sigma_r^g$ has support in $g^{-1}( r)$. Indeed,  for every $\varepsilon>0$ and $\varphi\in C_b(X)$ with support contained in $g^{-1}(-\infty, r-\varepsilon) \cup g^{-1}(  r+ \varepsilon , +\infty)$, the function $  F_{\varphi}$ is constant in $(r-\varepsilon, r+\varepsilon)$, and therefore 
$F_{\varphi}'( r) =0$. By \eqref{e15d}, $\int_X \varphi \,d\sigma_r^g=0$. So, the support of $\sigma_r^g$ is contained in $\cap_{\varepsilon >0}g^{-1}[r-\varepsilon, r+\varepsilon] = g^{-1}( r)$.  }

{\em If $g$ is not continuous, the existence of $\varphi\in C_b(X)$ with support contained in $g^{-1}(-\infty, r-\varepsilon) \cup g^{-1}(  r+ \varepsilon , +\infty)$ is not guaranteed, and this argument does not work. However, the argument in Remark 3.6 of \cite{DaLuTu14} shows that $\sigma^g_r = q_1(r) m_r$ for a.e. $r\in \R$ such that $q_1(r)>0$, where $m_r$ are the measures used in the proof of Proposition \ref{p2n}. Since the support of $m_r$ is contained in $g^{-1}(r)$ for almost all $r\in \R$, the support of $\sigma^g_r$ is contained in $g^{-1}(r)$ for almost all $r\in \R$ with $q_1(r)>0$. }

{\em In the next section we will show that for every $r\in \R$  the support of $\sigma_r^g$ is contained in $g^{*-1}( r)$, for a suitable version of $g$ (Prop. \ref{supporto}). }
 \end{Remark}

Theorem \ref{costruzione} asserts that $\sigma^g_r$ is nontrivial iff $q_1(r)>0$. So, 
it is important to know whether $q_{1}(r)>0$.  An obvious sufficient condition for $q_{1}(r)>0$ (in view of the identity  $  q_{1}(r)=\int_{g^{-1}(r, +\infty)} M^*\Psi \,d\nu  $) is that    $\nu(g^{-1}(r, +\infty) )>0$, and $M^*\Psi \geq 0$ on $g^{-1}(r, +\infty)$, $M^*\Psi >0$ on a subset of $g^{-1}(r, +\infty)$ with positive measure.   However, this is not easy to check. 

In the Gaussian case, under reasonable assumptions on $g$ we have $q_1(r) >0$ if and only if  $r\in (\essinf g, \esssup g)$ (\cite[Lemma 3.9]{DaLuTu14}). 
The proof is not easily extendable to our general setting, and in the next 
proposition we use an argument from  \cite[Second Edition, Prop. 2.1.8]{Nualart}. We need a further hypothesis, 

\begin{Hypothesis}
\label{h3}
If  $\varphi\in W^{1,2}(X, \nu)$ and $ M_2\varphi =0$, then $\varphi $ is constant $\nu$-a.e.
\end{Hypothesis}

For Hypothesis \ref{h3} be satisfied, one needs that $R$ be  one to one. However, even in the case $R=I$, Hypothesis \ref{h3} is not obvious. If it holds, the Dirichlet form 
${\mathcal E}(\varphi, \psi) := \int_X \langle M_2\varphi, M_2\psi\rangle \,d \nu$ is called irreducible. 

Of course, a sufficient condition for Hypothesis \ref{h3} be satisfied, is  that a Poincar\'e inequality holds, namely that there exists $C>0$ such that 
\begin{equation}
\label{poincare}
\int_X \bigg(\varphi - \int_X\varphi \, d\nu\bigg)^2 d\nu \leq C\int_X \|M_2\varphi\|^2 d\nu, \quad \varphi\in W^{1,2}(X, \nu). 
\end{equation}
We shall use the following lemma.

\begin{Lemma}
\label{Le:caratteristica}
Let Hypothesis \ref{h3} hold. If $B$ is a Borel set such that $\one_B \in W^{1,2}(X, \nu)$, then either $\nu(B)=0$ or $\nu(B)=1$. 
\end{Lemma}
 \begin{proof} 
We follow   \cite[Prop. 1.2.6]{Nualart}.  Assume that $\one_B \in W^{1,2}(X, \nu)$ and
let  $\varphi\in C^\infty_c(\R)$ be such that
$$
\varphi(r)=r^2,\quad \forall\; r\in[0,1].
$$
Then $\varphi\circ  \one_B = \one_B$ and $\varphi' \circ  \one_B = 2\one_B$, since 
$$
\varphi'(\one_B(x))=\left\{\begin{array}{l}
\varphi'(1)=2,\quad\mbox{\rm if}\;x\in [0, 1]\\
\varphi'(0)=0,\quad\mbox{\rm if}\;x\notin [0, 1].
\end{array}\right.
$$
Now by the chain rule (Lemma \ref{prodotto-ChainRule}(ii))
$$
M_2(    \one_B  )=M_2(\varphi\circ   \one_B)=\varphi'(  \one_B)M_2(  \one_B)=2  \one_B M_2(  \one_B)$$
so that $M_2 (    \one_B  )=0$. By Hypothesis \ref{h3}, $ \one_B $ is constant a.e., and the conclusion follows.
 \end{proof}

\begin{Proposition}
\label{q_1>0}
Under Hypotheses \ref{h1},  \ref{h2},  \ref{h3}, assume in addition that $M^*_p\Psi \in W^{1,p}(X, \nu)$ for every $p>1$. Then
for every $r\in \R$ we have 
$$q_{1}(r) > 0 \; \Longleftrightarrow \;r\in (\essinf g, \esssup g). $$
\end{Proposition}
 \begin{proof}
 The function $F_1(r) = \nu \{ x: \; g(x)\leq r\}$ is continuously differentiable, and it is constant in $(-\infty,  \essinf g) $ and in $(\esssup g, +\infty)$. Therefore, for every $r\in (-\infty,  \essinf g] 
 \cup [\esssup g, +\infty)$ we have $F_1'(r) = q_1(r) =0$. 
 
 To prove the converse, let us fix  $r_0$ such that $q_1(r_0)= 0$. We shall show that  the characteristic function $\one_{\{g>r_0\}}$ belongs to $W^{1,2}(X, \nu)$.  
 We approach $\one_{\{g>r_0\}}$ by the functions  $\varphi_{\eps} $ defined by 
$$\varphi_{\eps}(x) = \left\{ \begin{array}{ll}
0, & g(x) < r_0-\eps, \\
\\
\ds \frac{1}{2\eps}(g(x)- (r_0-\eps)), & r_0-\eps\leq g(x)\leq r_0+\eps, 
\\
\\
1, & g(x)> r_0+\eps. \end{array}\right. $$
for $\eps >0$. By Lemma \ref{ChainRuleLipschitz}, $\varphi_{\eps} \in W^{1,p}(X, \nu)$ for every $p$, and 
$$M(\varphi_{\eps} ) = \frac{1}{2\eps}\one_{\{r_0-\eps \leq g\leq r_0+\eps \}}Mg. $$
To estimate $\|M(\varphi_{\eps} )\|_{L^2(X, \nu;X)}$, we preliminary show that $q_1'$ is H\"older continuous, and that  $q_1'(r_0) =0$.

 The H\"older continuity of $q_1'$ follows from the regularity assumption on $M^*\Psi$. 
Indeed,  by \eqref{e17} we have

$$q_{1}(r)= -\int_{\{g<r\}}  M^* \Psi  \,d\nu  = -F_{M^*\Psi}(r), \quad r\in \R. $$
By assumption, $M^* \Psi\in W^{1,p}(X, \nu)$ for every $p>1$, so that by Lemma \ref{l4} $q_1$ is differentiable, and 
$$q_1'(r) = -q_{M^* \Psi}(r), \quad r\in \R. $$
Still by Lemma  \ref{l4}, $q_1'$ is H\"older continuous, with any exponent $\alpha \in (0, 1)$.

 Let us prove that $q_1'(r_0)=0$. Since $q_1(r_0)=0$, 
 $\sigma^g_{r_0}(X) =0$, and by Theorem \ref{costruzione} we get $q_{\varphi}(r_0)=0$ for every $\varphi\in C_b(X)$. Approaching every $\varphi\in W^{1,p}(X, \nu)$ by a sequence of $C^1_b$ functions and using estimate \eqref{e18}, we obtain $q_{\varphi}(r_0)=0$ for every $\varphi\in W^{1,p}(X, \nu)$. 
In particular,
$$q_{M^*_p \Psi}(r_0)= -q_1'(r_0)=0. $$
 Therefore, for every $\alpha \in (0,1) $ there exists $C_{\alpha}>0$ such that $|q_1(\xi)| \leq C_{\alpha}|\xi-r_0|^{1+\alpha}$, for every $\xi\in \R$. It follows that there exists $K_{\alpha} >0$ such that  for every $\eps >0$ we have
$$\bigg| \int_{r_0-\eps}^{r_0+\eps} q_1(\xi)d\xi\bigg| \leq K_{\alpha} \eps^{2+\alpha}. $$

Fix now  $\alpha \in (0, 1)$ and take $p = 2+ 4/\alpha$, so that $(\alpha +2)(p-2)/p = 2$. 
By the H\"older inequality we have
$$\int_X \|M\varphi_{\eps} \|^2 d\nu = \frac{1}{(2\eps)^2}\int_X \|Mg\|^2 \one_{\{r_0-\eps \leq g\leq r_0+\eps\}}\,d\nu$$
$$\leq \frac{1}{(2\eps)^2} \bigg(\int_X  \|Mg\|^p d\nu\bigg)^{2/p} \bigg( \int_{r_0-\eps}^{r_0+\eps} q_1(\xi)d\xi\bigg)^{(p-2)/p}$$
$$\leq \frac{1}{(2\eps)^2} \bigg(\int_X  \|Mg\|^p d\nu\bigg)^{2/p}(K_{\alpha} \eps^{2+\alpha})^{(p-2)/p} = C,  $$
with $C: = 2^{-2 } (\int_X  \|Mg\|^p d\nu) ^{2/p}K_{\alpha}^{(p-2)/p}$ independent of $\eps$. 

So, $\|\varphi_{\eps} \|_{W^{1, 2}(X, \nu)} $ is bounded by a constant independent of $\eps$. By Lemma \ref{prodotto-ChainRule}(v), $\one_{\{g>r_0\}}$ belongs to $W^{1,2}(X, \nu)$. 
By Lemma \ref{Le:caratteristica}, the measure of the set $\{ x: \; g(x) > r_0\}$ is either $0$ or $1$, namely $r_0\geq \esssup g$ or $r_0\leq \essinf g$ . 
 \end{proof}
   

 \subsection{Integration by parts formulae}


We recall that by Lemma \ref{prodotto-ChainRule}(i), the product $\varphi \psi$ belongs to $W^{1,p}(X, \nu)$ provided 
$\varphi\in W^{1,p_1}(X, \nu)$, $\psi\in W^{1,p_2}(X, \nu)$, with $1/p_1 + 1/p_2 \leq  1/p$. In this case, 
 for all  $F\in D(M_{p}^*)$ we may apply formula  \eqref{e1m} with $\varphi \psi$ replacing $\varphi$, and  we obtain
\begin{equation}
\label{e19n}
\int_X \langle M_{p_1}\varphi,F  \rangle\,\psi\,d\nu+\int_X \langle M_{p_2}\psi,F  \rangle\,\varphi\,d\nu=\int_X \varphi\,\psi\,M_{p}^*(F)\,d\nu .
\end{equation}

The following proposition is a first basic step towards an integration by parts formula.

 \begin{Proposition}
\label{partirozza}
 Assume that  Hypotheses \ref{h1} and  \ref{h2} are fulfilled.  Let  $p>1$,   $F\in  D(M^*_p)$, $\varphi \in   W ^{1,p_1}(X,\nu)$ for some $p_1>p$,   and 
assume that $ \langle Mg,F  \rangle\,\varphi$ belongs to $C_b(X)$ or to   $W ^{1,q}(X,\nu)$ for some $q>1$. 
 Then
\begin{equation}
\label{partiprima}
\int_{\{g < r\}} \langle M_{p_1}\varphi,F\rangle\,d\nu =  \int_{\{g<r\}}    \varphi\,M^*_p(F)\, d\nu + q_{ \langle Mg,F\rangle \varphi}(r), \quad r\in \R. 
\end{equation}
\end{Proposition}
\begin{proof}
For any $\varepsilon>0$ we set
\begin{equation}
\label{e21n}
\theta_\varepsilon(\xi)=\left\{\begin{array}{l}
1,\quad\mbox{\rm if}\;\xi\le r-\varepsilon,\\
\\
\ds -\frac{1}{\varepsilon}( \xi -r) ,\quad\mbox{\rm if}\; r-\varepsilon< \xi< r,\\
\\
0,\quad\mbox{\rm if}\;\xi\ge r ,
\end{array}\right.
\end{equation}
so that
\begin{equation}
\label{e22n}
\theta'_\varepsilon(\xi)=\left\{\begin{array}{l}
0,\quad\mbox{\rm if}\;\xi < r-\varepsilon,\\
\\
\ds -\frac1\varepsilon,\quad\mbox{\rm if}\; r-\varepsilon< \xi< r ,\\
\\
0,\quad\mbox{\rm if}\;\xi > r .
\end{array}\right.
\end{equation}
By Lemma \ref{ChainRuleLipschitz} (applied to $-g$), the composition $\theta_\varepsilon \circ g$ belongs to $W^{1,p_2}(X, \nu)$ for every $p_2>1$, and 
$$
M(\theta_\varepsilon \circ g) = (\theta'_\varepsilon \circ g) M(g). 
$$
Since $p_1>p$, choosing $p_2$ large enough we have $1/p_1 + 1/p_2 \leq  1/p$, and we may use 
 formula \eqref{e19n} with $\psi = \theta_\varepsilon \circ g$, to obtain
\begin{equation}
\label{e23n}
\int_X \langle M\varphi,F  \rangle\,(\theta_\varepsilon \circ g)\,d\nu = \frac{1}{\varepsilon}\int_{\{r-\varepsilon\le g\le r \}} \langle Mg,F  \rangle\,\varphi\,d\nu
+ \int_{X} \varphi\,(\theta_\varepsilon \circ g)\,M^*_p(F)\,d\nu.
\end{equation}
 Since $ \langle Mg,F  \rangle\,\varphi\in  C_b(X)\cup W^{1,q}(X,\nu)$, by Lemma \ref{l4} or by Theorem \ref{costruzione}  we have
$$\lim_{\varepsilon\to 0} \frac{1}{\varepsilon}\int_{\{r-\varepsilon\le g\le r \}} \langle Mg,F  \rangle\,\varphi\,d\nu = q_{\langle Mg,F  \rangle\,\varphi}(r).$$
On the other hand, $\theta_{\eps} \circ g$ converges a.e. to $\one_{ \{g\leq r\}}$, and by the Dominated Convergence Theorem we have
 $$\lim_{\varepsilon\to 0}\int_X \langle M\varphi,F  \rangle\,(\theta_\varepsilon \circ g)\,d\nu  = \int_{ \{ g\le r \}} \langle M\varphi,F  \rangle\,d\nu, $$
$$\lim_{\varepsilon\to 0}\int_X \varphi\,(\theta_\varepsilon \circ g)\,M^*_p(F)\,d\nu =  \int_{ \{ g\le r \}} \varphi\, M^*_p(F)\,d\nu .  $$
The conclusion follows.
 \end{proof}

Note that by \eqref{e15d}, if $ \langle Mg,F  \rangle\,\varphi\in C_b(X) $ then $q_{ \langle Mg,F\rangle \varphi}(r) $ is just the   integral of 
$ \langle Mg,F\rangle \varphi$ with respect to $\sigma_r^g$, and \eqref{partiprima} may be rewritten as
\begin{equation}
\label{parti_noi}
\int_{\{g<r\} } \langle M\varphi,F\rangle\,d\nu =   \int_{\{g<r\} }    \varphi\,M^*_p(F)\, d\nu + \int_X  \langle Mg,F\rangle \varphi \, d \sigma_r^g. 
\end{equation}

To improve formula \eqref{parti_noi} and extend it to a wider class of functions we have to work a bit. To this aim, in the next section we introduce the $p$-capacity and then we use it as a tool.

\section{$p$-capacities}

\begin{Definition}
Let  Hypothesis \ref{h1} hold. For every open set $O\subset X$ and $p>1$ we define the $p$-capacity of $O$ by 
$$C_p(O) := \inf\{ \|f\|_{W^{1,p}(X, \nu)}:\; f(x) \geq \one_O \; \nu-a.e.,  \; f\in W^{1,p}(X, \nu)\}. $$
If $B$ is any Borel set, we define 
$$C_p(B) := \inf\{C_p(O): \; O\; \text{is open}, \; O \supset B\}. $$

A function $f:X\mapsto \R$ is called $C_p$-quasicontinuous if for every $\eps >0$  there is an open set $O$ such that $C_p(O) <\eps $ and  $f$ is continuous in $X\setminus O$. 
\end{Definition}

This is just  Definition 8.13.1 of \cite{Boga2}, with the choice $\mathcal F = W^{1,p}(X, \nu)$. It follows immediately from the definition that for every Borel sets $A$, $B$ we have
$$C_p(A\cup B) \le C_p(A) + C_p(B), \quad C_p(B) \geq (\nu(B))^{1/p}. $$

We recall some properties of the $p$-capacity, taken from \cite[Sect. 8.13]{Boga2}.

\begin{Proposition}
\label{capacita}
\begin{itemize}
\item[(i)] Every element $f\in W^{1,p}(X, \nu)$ has a $C_p$-quasicontinuous version $f^*$, which satisfies
$$C_p(\{ x: \; f^*(x) >r\}) \leq \frac{1}{r}\| f\|_{W^{1,p}(X, \nu)} , \quad r>0.$$
\item[(ii)] Let  $(f_n)$be a sequence that  converges to $f$ in $W^{1,p}(X, \nu)$. For every $n$ let $f_n^*$ be any $C_p$-quasicontinuous version of $f_n$. Then there is 
a subsequence  $(f_{n_k}^*)$ that converges pointwise to $f^*$, except at most on a set with null $p$-capacity. 
\item[(iii)] If $O$ is an open set, $f $ is $C_p$-quasicontinuous and $f(x)\geq 0$ for $\nu$-a.e. $x\in O$, then $f(x)\geq 0$ in $ O$, except at most on a set with null $p$-capacity. 
\end{itemize}
\end{Proposition}

We are ready to exhibit a class of sets that are negligible with respect to all the  measures $\sigma_r^g$ constructed in Section 3.

\begin{Proposition}
\label{capacity}
Under Hypotheses \ref{h1} and \ref{h2}, let $B\subset X$ be a Borel set with $C_{p}(B) = 0$ for some $p>1$. Then $\sigma_r^g(B)=0$, for every $r\in \R$. 
\end{Proposition} 
\begin{proof}
For every $\eps >0$ let  $O_{\eps}\supset B$ be an open set such that $C_{p}(O_{\eps})<\eps$. Then there exists $f_{\eps}\in W^{1,p}(X, \nu)$ such that 
$\|f_{\eps}\|_{W^{1,p}(X, \nu)} \leq \eps$, $f_{\eps}\geq 0$ $\nu$-a.e.,  and $f_{\eps}\geq 1$ $\nu$-a.e.  in $O_{\eps}$.   
Let us fix an increasing sequence $(\theta_n)\subset C_b(X)$  that converges to $\one_{O_{\eps}}$ pointwise. For instance, we can take
$$\theta_n(x) =\left\{ \begin{array}{ll}
0, & x\in X\setminus O_{\eps}, 
\\
\\
n\,{\rm dist}(x, X\setminus O_{\eps}), & 0<{\rm dist}(x, X\setminus O_{\eps})< 1/n, 
\\
\\
1, & {\rm dist}(x, X\setminus O_{\eps})\geq 1/n
\end{array}\right. $$
Then, $\lim_{n\to \infty} \theta_n(x)= \one_{O_{\eps}}(x)$, for every $x\in X$. Using  the Dominated Convergence Theorem, and then formula \eqref{e15d},  we get 
\begin{equation}
\label{cap}
\sigma_r^g(O_{\eps}) = \int_{X} \one_{O_{\eps}} \,d\sigma_r^g = \lim_{n\to \infty} \int_{X}  \theta_n \,d\sigma_r^g  =  \lim_{n\to \infty} q_{ \theta_n}( r). 
\end{equation}
On the other hand, $f_{\eps}(x) \geq \one_{O_{\eps}}(x) \geq  \theta_n (x) $, for $\nu$-a.e. $x\in X$, so that the function
$F_{f_{\eps}- \theta_n }$ is increasing. In particular, $F_{f_{\eps}- \theta_n }'( r) =  q_{f_{\eps}}( r) - q_{ \theta_n}( r) \geq 0$ for every $r\in \R$ and $n\in \N$. 
Therefore, \eqref{cap} yields  
$$\sigma_r^g(O_{\eps}) \leq  q_{f_{\eps}} ( r ), \quad r\in \R. $$
On the other hand, by  \eqref{e18} we have 
$$|q_{f_{\eps}} ( r )| \leq K_p\|f_{\eps}\|_{W^{1,p}(X, \nu)}\leq K_p\eps ,$$
with $K_p$ independent of $\eps$. Therefore, $\sigma_r^g(O_{\eps}) \leq  K_p\eps $ for every $\eps >0$, which implies  $\sigma_r^g(B) =0$. 
\end{proof}

Now we extend formula \eqref{e15d} to Sobolev functions. The procedure is similar to \cite{CeLu14}, where Gaussian measures were considered. 

\begin{Theorem}
\label{est}
Let Hypotheses \ref{h1} and  \ref{h2} hold, and let $\varphi \in W^{1,p}(X, \nu)$ for some $p>1$. Fix any $r\in \R$. There exists a unique $\psi\in L^1( X, \sigma^g_r)$ such that   every sequence of $C^1_b$ functions $(\varphi_n)$ that converges to 
$\varphi$ in $W^{1,p}(X, \nu)$,  also converges in $L^1( X, \sigma^g_r)$ to   $\psi$. Setting  $T\varphi := \psi $, we have
\begin{equation}
\label{e15e}
F_{\varphi}'(r)   =\int_{X} T\varphi (x)\,\sigma_r^g(dx) .
\end{equation}
Moreover, 
\begin{itemize}
\item[(i)] $(\varphi_n)$ converges to $\psi$ in $L^q( X, \sigma^g_r)$, and $T\in \mathcal{L} ( W^{1,p}(X, \nu), L^q( X, \sigma^g_r))$ for every $q\in [1, p)$;
 \item[(ii)] for every $p$-quasicontinuous version $\varphi ^*$ of $\varphi $, we have $T\varphi (x)= \varphi ^*(x)$ for $\sigma_r^g$-a.e. $x\in X$. In particular, if $\varphi$ is continuous, then $T\varphi (x)= \varphi (x)$ for $\sigma_r^g$-a.e. $x\in X$;
 \item[(iii)] If  $\varphi_1 \in W^{1,p_1}(X, \nu)$, $\varphi_2 \in W^{1,p_2}(X, \nu)$, and $1/p_1 + 1/p_2 <1$, then $T(\varphi_1\varphi_2) = T(\varphi_1) T(\varphi_2)$ (as elements of $L^1(X, \sigma^g_r$). 
 \end{itemize}
\end{Theorem}
\begin{proof}  By  \eqref{e15d}, for every $\varphi\in C^1_b(X)$ we have
$$
q_{|\varphi  |}( r)  =\int_{X} |\varphi  |\,\sigma_r^g(dx)  . 
$$
Take $\varphi = \varphi _n -\varphi_m$. By Corollary \ref{partepositiva},  $|\varphi _n -\varphi_m|\in W^{1,p}(X, \nu)$, and $\lim_{n, m\to \infty}\|\, |\varphi _n -\varphi_m|\,\|_{W^{1,p}(X, \nu)} =0$. 
Using   estimate 
\eqref{e18} we get 
\begin{equation}
\label{e15f}
q_{|\varphi _n -\varphi_m|}( r) = \int_{X} |\varphi _n -\varphi_m|\,\sigma_r^g(dx)  \leq K_p\|\,|\varphi_n - \varphi_m\,\|_{W^{1,p}(X, \nu)} . 
\end{equation}
Therefore, $(\varphi_n )$ is a Cauchy sequence in $L^1( X, \sigma^g_r)$ and it converges to a limit $\psi$ in $L^1( X, \sigma^g_r)$. The limit function $\psi $ is apparently the same for all  sequences that converge to 
$\varphi$ in $W^{1,p}(X, \nu)$. Indeed, if $\varphi_n \to \varphi$, $\widetilde{\varphi}_n \to \widetilde{\varphi}$ in $W^{1,p}(X, \nu)$ as $n\to \infty$, the difference $\varphi_n -\widetilde{\varphi}_n$ vanishes in $L^1(X, \sigma_r^g)$  by estimate \eqref{e18} with $\varphi _n -\varphi_m$ replaced by $\varphi_n -\widetilde{\varphi}_n$. 

Still by estimate \eqref{e18}, the sequence $(q_{\varphi_n} ( r))$ converges to $q_\varphi ( r)$, and \eqref{e15e} follows. 

\vspace{3mm}

To  prove statement (i) we follow the above procedure, replacing $|\varphi _n -\varphi_m|$ with  $|\varphi _n -\varphi_m|^q$ that belongs to $C^1_b(X)$ for $q>1$, and vanishes in $W^{1, p/q}(X, \nu)$ as $n$, $m\to +\infty$ by     \eqref{varphi^s}. By estimate \eqref{e18} we have  
$$q_{|\varphi _n -\varphi_m|^q}( r)  \leq K_{p/q}\|\,|\varphi_n - \varphi_m|^q\,\|_{W^{1,p/q}(X, \nu)} , $$
so that $(\varphi_n )$ is a Cauchy sequence in $L^q( X, \sigma^g_r)$ and its $L^1( X, \sigma^g_r)$-limit $\psi$ belongs to $L^q( X, \sigma^g_r)$. 

\vspace{3mm}

Let us prove (ii). By Proposition \ref{capacita}(ii), a subsequence  $(\varphi_{n_k})$ converges to $\varphi^*(x)$  for every $x\in X$ except at most on a set with zero  $p$-capacity. By   Proposition \ref{capacity}, such a subsequence converges $\sigma_r^g$-a.e to $\varphi^* $. By the first part of this proposition, $(\varphi_{n_k})$ converges to $T\varphi$ in $L^1(X, \sigma_r^g)$. A further  subsequence of $(\varphi_{n_k})$ converges to $T\varphi$, $\sigma_r^g$-a.e. Therefore, $T\varphi = \varphi^*$, $\sigma_r^g$-a.e.

\vspace{3mm}

To prove (iii) it is enough to approach $\varphi_1$ and $\varphi_2$  by  sequences $(\varphi_{1,n})$, $(\varphi_{2,n})$ of $C^1_b$ functions, in $W^{1, p_1}(X, \nu)$, $W^{1, p_2}(X, \nu)$, respectively. The product $\varphi_{1,n} \varphi_{2,n}$ converges to $\varphi_1\varphi_2$ in $W^{1,s}(X, \nu)$, for $s= (p_1+p_2)/p_1p_2$, therefore $T(\varphi_1\varphi_2) = \lim_{n\to \infty} \varphi_{1,n} \varphi_{2,n}$ in $L^1( X, \sigma_r^g)$. On the other hand, $T(\varphi_1) = \lim_{n\to \infty} \varphi_{1,n}$ in $L^q( X, \sigma_r^g)$ for every $q<p_1$, 
$T(\varphi_2) = \lim_{n\to \infty} \varphi_{1,n}$ in $L^r( X, \sigma_r^g)$ for every $r<p_2$. Choosing $q<p_1$ and $r<p_2$ such that $1/q+1/r=1$, we obtain $T(\varphi_1)T(\varphi_2) = \lim_{n\to \infty} \varphi_{1,n} \varphi_{2,n}$ in $L^1( X, \sigma_r^g)$, and the statement follows. 
 \end{proof}

The results that we have proved up to now are independent of the version of $g$ that we have considered. Instead, from now on 
  we fix a $p$-quasicontinuous version $g^*$ of $g$, for some $p>1$. This is because we shall consider the $\nu$-negligible sets $(g^*)^{-1}(r)$  for $r\in \R$.

With the aid of Theorem \ref{est} we can study the  supports of the measures $\sigma_r^g$.

\begin{Proposition}
\label{supporto}
For every $r_0\in \R$, the support of $\sigma_{r_0}^g$ is contained  in $g^{*-1}(r_0)$. 
\end{Proposition}
\begin{proof}
Fix $\eps >0$, and set  $A:= g^{*-1}(-\infty, r_0-\eps) \cup g^{*-1}(r_0+\eps ,  \infty )$. Our aim is to show that
\begin{equation}
\label{A}
\int_X \one_A d\sigma_{r_0}^g =0, 
\end{equation}
which implies that the support of  $\sigma_{r_0}^g$ is contained  in $g^{*-1}([r_0-\eps, r_0+\eps])$. Since $\eps$ is arbitrary, the statement will follow. 

We approach $\one_{(-\infty, r_0-\eps) \cup(r_0+\eps , + \infty )}$ by a sequence of Lipschitz functions, 
$$\chi_n(\xi)=\left\{ \begin{array}{ll}
1, & \xi\leq r_0-\eps -1/n, \\
-n(\xi -( r_0-\eps)), & r_0-\eps -1/n\leq \xi\leq r_0-\eps , 
\\
0, & r_0-\eps \leq \xi\leq r_0+\eps, 
\\
n(\xi-(r_0+\eps)), & r_0+\eps\leq \xi \leq r_0+\eps+ 1/n, 
\\
1, & \xi \geq r_0+\eps. 
\end{array}\right. $$
We have $\lim_{n\to \infty} \chi_n(\xi)= \one_{(-\infty, r_0-\eps) \cup(r_0+\eps , + \infty )}(\xi)$, for every $\xi \in \R$. Consequently, 
$\chi_n\circ g^*$ converges pointwise, for every $x\in X$, to $ \one_A$. Since $0\leq \chi_n\circ g^*\leq 1$, by the Dominated Convergence Theorem we get
\begin{equation}
\label{limite}
\int_X \one_A d\sigma_{r_0}^g = \lim_{n\to \infty}\int_X \chi_n\circ g^*\, d\sigma_{r_0}^g. 
\end{equation}
For every  $n$, $ \chi_n\circ g\in W^{1,p}(X, \nu)$, by Lemma \ref{ChainRuleLipschitz}. By Lemma \ref{l4}, $q_{\chi_n\circ g}$ is continuous, so that the function $F_{\chi_n\circ g}( r) = \int_{g^{-1}(-\infty, r)}\chi_n\circ g\,d\mu$, whose derivative is $q_{\chi_n\circ g}$, is $C^1$. By the definition of  $\chi_n$,   $F_{\chi_n\circ g}$ is constant, equal to  $F_{\chi_n\circ g}(r_0-\eps)$, in the interval  $[r_0-\eps, r_0+\eps]$, so that the derivative  $q_{\chi_n\circ g}$ vanishes  in $(r_0-\eps, r_0+\eps)$. In particular, it vanishes at $r_0$.  

By \eqref{e15e} we have
$$q_{\chi_n\circ g}(r_0) = \int_X  T(\chi_n\circ g)\, d\sigma_{r_0}^g , $$
where $T$ is the  operator defined in Theorem \ref{est}. Since  $g^*$ is $p$-quasicontinuous and $\chi_n$ is continuous,  $\chi_n\circ g^*$ is $p$-quasicontinuous. It coincides with 
$\chi_n\circ g$ outside a $\nu$-negligible set, therefore it is a $p$-quasicontinuous version of $\chi_n\circ g$. 
By Theorem \ref{est},  $T(\chi_n\circ g)$ coincides with  $\chi_n\circ g^*$, up to $\sigma_{r_0}^g$-negligible sets. Therefore, for every  $n\in \N$, 
$$0= q_{\chi_n\circ g}(r_0) = \int_X T(\chi_n\circ g)\, d\sigma_{r_0}^g = \int_X \chi_n\circ g^*\, d\sigma_{r_0}^g,  $$
and  \eqref{A} follows from  \eqref{limite}. 
 \end{proof}

 Proposition \ref{supporto} justifies the following definition.

 \begin{Definition}
 Let $\varphi\in W^{1,p}(X, \nu)$ for some $p>1$, and let $r\in \R$. We define the trace of $\varphi$ at $g^{*-1}(r)$ as the function $T\varphi $ given by Theorem \ref{est}. 
 \end{Definition}

Characterizing the range of the trace operator is a difficult problem, that is out of reach for the moment. In the case of Gaussian measures in Banach spaces the range of the trace has been characterized only for $g\in X^*$ (\cite{CeLu14}). Even worse, for very smooth functions in Hilbert spaces such as $g(x) = \|x\|^2$ we do not know whether the traces of elements of $W^{1,p}(X, \nu)$ belong to $L^p(X, \sigma^g_r)$ with the same $p$. See the discussion in \cite{CeLu14}.

 Now we read again formula \eqref{partiprima} in terms of surface integrals.

\begin{Corollary}
\label{parti_generale}
Assume that  Hypotheses \ref{h1} and  \ref{h2} are fulfilled.  Let  $p>1$,   $F\in  D(M^*_p)$, $\varphi \in   W ^{1,p_1}(X,\nu)$ for some $p_1>p$,   and 
assume that   $ \langle Mg,F  \rangle\,\varphi$ belongs to $C_b(X)$ or to $W ^{1,q}(X,\nu)$ for some $q>1$. 
Then for every $r\in \R$
\begin{equation}
\label{parti_noi_generale}
\begin{array}{lll}
\ds \int_{\{g<r\} } \langle M_{p_1}\varphi,F\rangle\,d\nu & = & \ds \int_{\{g<r\}}    \varphi\,M^*_p(F)\, d\nu + \int_X T( \langle Mg,F\rangle \varphi )\, d \sigma_r^g
\\
\\
& = & \ds \int_{\{g<r\}}    \varphi\,M^*_p(F)\, d\nu + \int_{g^{*-1}( r)} T( \langle Mg,F\rangle \varphi )\, d \sigma_r^g. 
\end{array}
\end{equation}
\end{Corollary}
\begin{proof}
By formula \eqref{e15e} for every $r\in \R$ we have
$$\int_{g^{*-1}( r)}T(  \langle Mg,F\rangle \varphi ) \,d\sigma_r^g = \int_{X}T(  \langle Mg,F\rangle \varphi ) \,d\sigma_r^g  = q_{ \langle Mg,F\rangle \varphi }( r).  $$
On the other hand,   Proposition \ref{partirozza}  yields
$$q_{ \langle Mg,F\rangle \varphi }( r)=  \int_{\{g<r\}} \langle M_{p_1}\varphi, F\langle \,d\nu    -\int_{\{g<r\}}    \varphi\,M^*_p(F)\, d\nu .  $$
and the statement follows. 
\end{proof}

Since  the operators $M_p$ and $M_p^*$ play the role of the gradient and of the negative divergence, formula  \eqref{parti_noi_generale} is a version of the Divergence Theorem  in our context. The similarity gets better if we assume that   $\|Mg\|\in W^{1,q}(X, \nu)$ for every $q>1$. In this case, recalling Theorem \ref{est}(iii)
we may rewrite \eqref{parti_noi_generale}  as 
\begin{equation}
\label{Th:divergenza}
\begin{array}{lll}
\ds \int_{\{g<r\} } \langle M_{p_1}\varphi,F\rangle\,d\nu & = & \ds \int_{\{g<r\} }    \varphi M^*_pF\, d\nu + \int_{X} T( \langle \frac{Mg}{\|Mg\|},F\rangle \varphi ) T(\|Mg\| )\, d \sigma_r^g
\\
\\
& = &  \ds \int_{\{g<r\} }    \varphi M^*_pF\, d\nu + \int_{X} T( \langle \frac{Mg}{\|Mg\|},F\rangle \varphi )  \, d \rho_r,  
\end{array}
\end{equation}
where 
$$\rho_r (dx): = T(\|Mg\|)(x)  \sigma_r^g(dx), $$
 so that  $Mg/\|Mg\|  $ plays the role of the exterior normal vector to the surface $g^{*-1}(r)$, and the weighted measure $\rho_r $ plays the role of normalized surface measure. In fact $\rho_r $ is a distinguished surface measure and it will be discussed in the next section.

Let us consider now the case of constant vector fields $F$. Namely, we fix $z\in X$ and we assume that $F_z(x) \equiv z$ belongs to $D(M_p^*)$ for some $p>1$. 
We recall that $F_z\in D(M^*_p)$ iff there exists $C_{p,z}>0$ such that 
\begin{equation}
\label{verificabile}
\bigg| \int_X \langle R\nabla \varphi, z\rangle d\nu\bigg| \leq C_{p,z} \|\varphi\|_{L^p(X, \nu)}, \quad \varphi \in C^1_b(X) 
\end{equation}
(see the Introduction). In this case,  
we set $v_z: = M^*_p(F_z)$ and we rewrite \eqref{parti_noi_generale} for every $\varphi\in W^{1,p}(X, \nu)$ as 
\begin{equation}
\label{parti_noi_z}
\int_{g^{-1}(-\infty, r) } \langle M_p\varphi,z\rangle\,d\nu = \int_{g^{-1}(-\infty, r) }    \varphi\,v_z\, d\nu + \int_{X} T( \langle Mg,z\rangle \varphi )\, d \sigma_r^g. 
\end{equation}
provided $\langle Mg,z\rangle \varphi $ belongs to $C_b(X)$ or to $W^{1,q}(X, \nu)$ for some $q>1$.

\section{Dependence on $g$: comparison with the geometric measure theory approach}

Even for continuous or smooth $g$, the measures $\sigma_r^g$ constructed  in the previous sections depend explicitly on the defining function $g$, and not only on the sets $g^{-1}(r)$ or $g^{-1}(-\infty, r)$.  
In particular, if we replace $g$ by $\widetilde{g} = \theta \circ g$ by a smooth $\theta : \R\mapsto \R$ such that $\inf \theta' >0$, it is easy to check that $\widetilde{g} $ satisfies Hypothesis \ref{h2}, and  using the definition we see that for every $r\in \R$, setting $\Sigma := g^{-1}(r) =  \widetilde{g} ^{-1}(\theta(r))$ we have
$$\int_{\Sigma} \varphi \,d\sigma^g_r = \theta'(r) \int_{\Sigma} \varphi \,d\sigma^{\widetilde{g}}_{\theta(r )}, \quad \varphi \in C_b(X). $$
So, it is desirable to modify the  construction of our surface measures in order to get rid of the dependence on $g$, and to get a surface measure with some intrinsic analytic or geometric properties.

In the case of Gaussian measures in Banach spaces, for suitably smooth hypersurfaces $g^{-1}(r)$ the measure $|\nabla_Hg|_H \sigma_r^g$, where $H$ is the Cameron-Martin space,  is independent of $g$, and it coincides with the restriction of  the Hausdorff--Gauss measure of Feyel and de La Pradelle (\cite{FePr92}) to the hypersurface, and with the perimeter measure relevant to the set $\Omega = g^{-1}(-\infty, r)$ from the geometric measure theory in abstract Wiener spaces (\cite{F,FH,AMMP}). See \cite{CeLu14}. 

In our setting, what plays the role of $|\nabla_Hg|_H $ is  $\|Mg\|$. We shall show that $\|Mg\| \sigma_r^g$ depends on $g$ only through the set $g^{-1}(-\infty, r)$, among a class of good enough $g$, and it is a sort of perimeter measure.

As a first step, we notice that if $\|Mg\|\in W^{1,q}(X, \gamma)$ for some $q>1$, then $Mg/\|Mg\|  \in D(M_s^*)$ for every $s>q/(q-1)$. This comes from Lemma \ref{Le:divergenza}, writing 
$$\frac{Mg}{\|Mg\|} = \frac{Mg}{\|Mg\|^2} \|Mg\|, $$ 
and recalling that $Mg/\|Mg\|^2 \in D(M_p^*)$ for every $p>1$, by Hypothesis \ref{h2}. Lemma  \ref{Le:divergenza} also yields
$$M_s^* \bigg( \varphi \frac{Mg}{\|Mg\|} \bigg) = \varphi \|Mg\| M_p^*\bigg(  \frac{Mg}{\|Mg\|^2}\bigg) - \langle M_q( \varphi  \|Mg\|) ,  \frac{Mg}{\|Mg\|^2}\rangle , $$
for every $\varphi \in C^1_b(X)$. Comparing with \eqref{e17}, we obtain
\begin{equation}
\label{startingpoint}
q_{\|Mg\|\varphi}(r) = -\int_{\{ g  <r\} } M_s^*\bigg( \varphi \frac{Mg}{\|Mg\|}   \bigg)\,d\nu, \quad r\in \R. 
\end{equation}
and by Theorem \ref{est},  
\begin{equation}
\label{startingpoint2}
\int_{g^{*-1}(r)} \varphi \,T(\|Mg\|) d\sigma^g_r = -\int_{\{ g  <r\} } M_s^*\bigg( \varphi \frac{Mg}{\|Mg\|}   \bigg)\,d\nu, \quad r\in \R. 
\end{equation}
%
%
%
The right-hand side of \eqref{startingpoint} and of \eqref{startingpoint2} is the negative  integral over $g^{-1}(-\infty, r)$ of $M_s^*(\varphi F)$, where $F= Mg/\|Mg\|$ plays the role of the exterior unit normal vector to the level surfaces of $g$.  It is indeed the   exterior unit normal vector to $\partial \{ x: \; g(x)<r\}$ if $g$ is smooth enough and $R=I$. 

To go on, it is convenient to introduce spaces of $W^{1,p}$ vector fields. 

\begin{Definition}
For every $p>1$ we denote by $W^{1,p}(X, \nu;X)$ the space of vector fields $F:X\mapsto  X$ such that for a given orthonormal basis $\{ e_i:\; i\in \N\}$, the functions $f_i: \langle F, e_i\rangle \in W^{1,p}(X, \nu)$ for every $i\in \N$, and $(\sum_{i=1}^{\infty} \|M_pf_i\|^2) ^{1/2} \in L^p(X, \nu)$. 
\end{Definition}

It is easy to see that the definition does not depend on the chosen orthonormal basis. 
The  standard proof of the following lemma is left to the reader. 

\begin{Lemma}
\label{Le:D1p}
\begin{itemize}
\item[(i)] If $F_1\in W^{1,p_1}(X, \nu;X)$, $F_2\in W^{1,p_2}(X, \nu;X)$, with $1/p:= 1/p_1 + 1/p_2 <1$, then $x\mapsto \langle F_1(x), F_2(x)\rangle $ belongs to $W^{1,p}(X, \nu)$. %
\item[(ii)]  If $F\in W^{1,p_1}(X, \nu;X)$, $\varphi \in W^{1,p_2}(X, \nu)$, with $1/p:= 1/p_1 + 1/p_2 <1$, then $\varphi  F  $ belongs to $W^{1,p}(X, \nu)$. 
\item[(iii)] If $F\in W^{1,p}(X, \nu;X)$ for some $p>1$, then  $x\mapsto \|F(x)\|$ belongs to $W^{1,p}(X, \nu)$. 
\end{itemize}
\end{Lemma}

The following theorem is the main result of this section.

\begin{Theorem}
\label{Th:perimeter}
Let Hypotheses \ref{h1} and  \ref{h2} hold, and assume in addition that $Mg\in W^{1,q}(X, \nu;X)$ for some $q>2$. 
Then for every $\varphi \in C^1_b(X)$ with nonnegative values and for any  $t\in (q', q)$, $s> q'$ we have 
\begin{equation}
\label{misuraperimetro}
\int_X \varphi T(\|Mg\|)d\sigma^g_r =  \max \bigg\{ \int_{\{g <r\}  }M_s^*(\varphi F) \,d\nu: \; F\in   W^{1,t}(X) \cap D(M_s^*), \, \|F(x) \|\leq 1 \, {\rm a.e.}
 \bigg\}. 
\end{equation}
The maximum is attained at $F= -Mg/\|Mg\|$. 
\end{Theorem}
\begin{proof}
By Lemma \ref{Le:D1p}(iii), $\|Mg\|\in W^{1,q}(X, \nu)$ and therefore $\|Mg\|\varphi \in W^{1,q}(X, \nu)$. Then, by formulae \eqref{startingpoint} and \eqref{startingpoint2},  
$$ \int_X \varphi\, T(\|Mg\|)d\sigma^g_r  = q_{\|Mg\|\varphi}(r). $$
So, we have to show that $q_{\|Mg\|\varphi}(r)$ is  equal to the right-hand side of   \eqref{misuraperimetro}.  The proof is in two steps. 

In the first step we shall prove that the vector field 
 $F = Mg/\|Mg\|$  is one of the admissible vector fields in the right hand side of \eqref{misuraperimetro}, namely that it   belongs to $ D(M_s^*)$ for every $s> q'=q/(q-1)$ and to  $  W^{1,t}(X,\nu;X)$ for every $t\in (q', q)$.  
 
 In the second step we shall prove that for every admissible vector field $F$ in the right hand side of  \eqref{misuraperimetro}, the integral $ \int_{\{g<r\}} M_s^*(\varphi F)d\nu$ is equal to $q_{\langle Mg, F\rangle \varphi }(r) $ (of course, we need to show that $\langle Mg, F\rangle \varphi$ belongs to $W^{1,p}(X, \nu)$ for some $p$). Then, using the definition of $q_{\varphi}$,  it will be easy to see that $q_{\langle Mg, F\rangle \varphi }(r) \leq q_{\|Mg\|\varphi}(r)$ if $\varphi$ has nonnegative values. 
 
  In view of formula \eqref{startingpoint}, the statement will follow. 
  
  Throughout the proof we denote by $\{e_i:\; i\in \N\}$ any orthonormal basis of $X$. 
 
\vspace{3mm}

\noindent  {\em Step 1.}
 $F$ may be written as the product of $ Mg/\|Mg\|^2$, which is in $D(M_p^*)$ for every $p$ by Hypothesis  \ref{h2}, and the scalar function $\|Mg\|\in W^{1,q}(X, \nu)$ by Lemma \ref{Le:D1p} (iii). Lemma \ref{Le:divergenza}(i) implies that $F\in D(M_s^*)$ for $s= pq/(q-p)$, for every $p\in (1,q)$. Letting $p\to 1$, we obtain $F\in D(M_s^*)$ for every $s> q/(q-1)=q'$. 
 
Let us prove that $F  \in W^{1,t}(X,\nu;X)$, for every $t<q$. We have  $F= Mg \,\psi$, with $\psi = 1/\|Mg\|$. As easily seen approximating $\psi $ by 
$$\ds \psi_n(x) :=   \bigg(\ds\sum_{i=1}^n \langle Mg(x), e_i \rangle^2 + 1/n\bigg)^{-1/2}, $$
$\psi \in W^{1,p}(X, \nu)$ for every $p<q$,  and $M_p \psi = \|Mg\|^{-2} \sum_{k=1}^{\infty} M_q(\langle Mg, e_k\rangle )e_k$. Using Lemma \ref{Le:D1p}(ii), we obtain $F\in W^{1,t}(X, \nu;X) $ if $1/t = 1/p +1/q <1$, so that  $F\in W^{1,t}(X, \nu;X) $ for $t\in (1, q/2)$. To avoid this restriction 
we use the definition of the spaces $W^{1,t}(X, \nu;X)$ instead of Lemma \ref{Le:D1p}, and we take advantage  of $\|Mg\|^{-1}\in L^p(X, \nu)$ for every $p$, which is a consequence of Hypothesis \ref{h2} (see the introduction). Setting $f_i:= \langle F, e_i\rangle = \langle Mg, e_i\rangle /\|Mg\|$ for $i\in \N$, each $f_i$ belongs to $W^{1,p}(X, \nu)$ for $p\in (1, q)$ and
$$M_pf_i = \frac{M_q \langle Mg, e_i\rangle}{\|Mg\|} -  \langle Mg, e_i\rangle\frac{\sum_{k=1}^{\infty}(M_q\langle Mg, e_k\rangle ) \, e_k}{ \|Mg\|^2 }$$
so that 
$$ \|M_pf_i \| \leq \frac{\|M_q \langle Mg, e_i\rangle \|}{\|Mg\|} + \frac{|\langle Mg, e_i\rangle|}{ \|Mg\|^2}\bigg(\sum_{k=1}^{\infty}\|M_q\langle Mg, e_k\rangle \|^2\bigg)^{1/2}$$
which implies
$$\bigg( \sum_{i=1}^{\infty} \|M_pf_i \|^2 \bigg)^{1/2} \leq 2 \frac{\bigg(\sum_{i=1}^{\infty}\|M_q \langle Mg, e_i\rangle \|^2\bigg)^{1/2}}{\|Mg\|} . $$
Therefore, $F \in W^{1,t}(X,\nu;X)$, for every $t<q$.

\vspace{3mm}

\noindent  
{\em Step 2.}
Now we show that   if  $F\in   W^{1,t}(X,\nu;X) \cap D(M_s^*)$ for some $t >q'$, $s>q'$, is such that $\|F(x) \|\leq 1 $ $\nu$-a.e., then we have 
\begin{equation}
\label{OK}
\int_{\{g<r\}}   M_s^*(\varphi F) \,d\nu \leq  q_{\|Mg\|\varphi}(r) .
\end{equation}
To this aim, we prove that $\langle Mg, F\rangle \in W^{1,p}(X, \nu)$ for some $p>1$.  

Set $f_i(x):= \langle F(x), e_i\rangle$ for $i\in \N$, $x\in X$. Since  $|\langle Mg, F\rangle| \leq \|Mg\|$, $\langle Mg, F\rangle \in L^{q}(X, \nu)$ and the series
$s_n = \sum_{i=1}^n f_i \langle Mg, e_i\rangle $ converges to $\langle Mg, F\rangle$ in $ L^{q}(X, \nu)$. Let us prove that it converges in a Sobolev space. For every $i\in \N$, 
we have  $\langle Mg, e_i \rangle \in W^{1,q}(X, \nu)$, $f_i\in W^{1,t}(X, \nu)$ and  $t>q'$,  so that   $\langle Mg, e_i \rangle f_i\in W^{1,p}(X, \nu)$ with $p = qt/(q+t)$ by Lemma \ref{prodotto-ChainRule}(i). Moreover, 
$$M_{p}s_n = \sum_{i=1}^n   f_i M_q\langle Mg, e_i\rangle  + \sum_{i=1}^n \langle Mg, e_i\rangle M_{t}f_i, $$
so that 
$$\|M_{p}s_n\| \leq \bigg(\sum_{i=1}^n \|M_q\langle Mg, e_i\rangle  \|^2\bigg)^{1/2} + \bigg( \sum_{i=1}^n \|M_tf_i\|^2\bigg)^{1/2}\|Mg\|$$
and the  series $(M_{p}s_n)$ converges in $L^p(X, \nu;X)$. 
 Therefore,  $\langle Mg, F\rangle \varphi \in W^{1,p}(X, \nu)$ for every $\varphi\in C^1_b(X)$, and  Proposition \ref{partirozza} yields
$$\int_{\{g<r\}}  M_s^*(\varphi F) \,d\nu = - q_{\langle Mg, F\rangle \varphi}. $$
We recall now that if $\varphi_1 \leq \varphi_2$ a.e., then $q_{\varphi_1}(r) \leq q_{\varphi_2}(r)$, for every $r$. In our case, $\varphi $ has nonnegative values, so that 
 $\langle Mg(x), F(x)\rangle \varphi(x) = \langle Mg(x)/\|Mg(x)\|, F(x)\rangle \varphi (x)\|Mg(x)\| \leq  \varphi (x)\|Mg(x) \|$ for a.e. $x$, and therefore $q_{\langle Mg, F\rangle \varphi}(r)\leq q_{\|Mg\|\varphi}(r)$ and \eqref{OK} follows. 
\end{proof}

Let $g_1$, $g_2$ satisfy the assumptions of Theorem  \ref{Th:perimeter}, and assume that for some $r\in \R$ we have $\{x:\; g_1(x) <r\} = \{x:\; g_2(x) <r\} =: \Omega$. By Theorem \ref{est}, for every $\varphi\in C^1_b(X)$, 
$$q_{\|Mg_i\|\varphi}(r) = \int_X T_i(\varphi \|Mg_i\|) \,d \sigma_r^{g_{i}} =  \int_X \varphi T_i( \|Mg_i\|) \,d \sigma_r^{g_{i}}, \quad i=1, 2, $$
where $T_i$ is the trace of Sobolev functions in $L^1(X, \sigma_r^{g_{i}})$. 
%
%
If in addition $\varphi$ has nonnegative values, by Theorem \ref{Th:perimeter} the left hand side depends only on the set $\Omega$. Approximating  every nonnegative  $\varphi \in UC_b(X)$ by a sequence of nonne\-gative $C^1_b$ functions, we obtain  $ \int_X  \varphi T_1(\|Mg_1\| )\,d \sigma_r^{g_{1} }=  \int_X  \varphi T_2( \|Mg_2\|) \,d \sigma_r^{g_{2} }$; splitting every $\varphi \in UC_b(X)$ as $\varphi = \varphi^+ - \varphi^-$ we obtain $ \int_X  \varphi T_1(\|Mg_1\|) \,d \sigma_r^{g_{1}} =  \int_X  \varphi T_2(\|Mg_2\| )\,d \sigma_r^{g_{2} }$. This shows that the weighted measures $T_1( \|Mg_1\|) \,d \sigma_r^{g_{1}}$ and $T_2(\|Mg_2\| )\,d \sigma_r^{g_{2} }$ coincide.

Fix any $g$ satisfying the assumptions of Proposition \ref{Th:perimeter}, and define
\begin{equation}
\label{rho_r}
\rho_r (dx) := T( \|Mg  \|)  \sigma_r^g(dx). 
\end{equation}
Taking in particular $\varphi \equiv 1$, we get 
$$ \rho _r(g^{-1}(r) ) =  \sup \bigg\{ \int_{\{g<r\}  }M_s^*(  F) \,d\nu: \; F\in   W^{1,t}(X) \cap D(M_s^*), \, \|F(x) \|\leq 1 \, {\rm a.e.}\bigg \} <+\infty . $$
We recall that a bounded Borel set $\Omega \subset \R^n$ has finite perimeter if $\one_{\Omega}$ is a function with bounded variation, and in this case  the perimeter measure $m$  is defined as the total variation measure of $D\one_{\Omega}$. Equivalently, $\Omega$ has finite perimeter if and only if 
$$ \sup \bigg\{ \int_{\Omega } \text{div}\;F \,dx: \; F\in  C^1_c(\Omega, \R^n), \, \|F(x) \|\leq 1 \,\forall x\in \Omega \bigg \} <+\infty , $$
and in this case for every $\varphi\in C^1_b(\R^n)$ with nonnegative values we have 
$$\int \varphi\,dm = \sup \bigg\{ \int_{\Omega } \text {div}\;(F\varphi)  \,dx: \; F\in  C^1_c(\Omega, \R^n), \, \|F(x) \|\leq 1 \,\forall x\in \Omega \bigg \}, $$
to be compared  to formula  \eqref{misuraperimetro}. In our setting  the operators $-M_s^*$ play the role of the   divergence, the measure $\rho_r$ plays the role of the perimeter measure, and  $\rho_r (g^{-1}(r) )$ may be called the (generalized) perimeter of the set $g^{-1}(-\infty, r)$. The vector field $Mg/\|Mg\|$ plays the role of the exterior normal vector field at  $g^{-1}(r) $. 
%
%
It would be worth (although it is not the aim of this paper) to develop a theory of BV functions for general differentiable measures in Hilbert or Banach spaces, and to go on in the investigation of perimeter measures.

\section{Weighted Gaussian  measures} 

We refer to the paper \cite{Simone}, where weighted Gaussian measures in Banach spaces were studied. Let $\nu(dx) = w(x)\mu(dx)$, where $\mu $ is a centered nondegenerate Gaussian measure with covariance $Q$. The nonnegative weight $w$ satisfies 
\begin{equation}
\label{w}
w, \; \log w\in W^{1,s}(X, \mu) \;\; \forall s>1. 
\end{equation}
Of course, every $C^1$ weight with positive infimum and such that 
$w(x)$, $\|\nabla w(x)\|\leq C\exp(\alpha\|x\|)$ for some  $C$, $\alpha >0$  satisfy  assumption \eqref{w}. Examples of  discontinuous weights that satisfy  \eqref{w} are in \cite{Simone} (in the space $X = \ell ^2$) and in \cite{DaLu14} (in the space $X= L^2(0,1)$
 with respect to the Lebesgue measure). 

Since we are considering two different measures, $\mu$ and $\nu$, it is convenient to denote  by $M_p^{\mu}$, $M_p^{\nu}$ the operators obtained by our procedure using the measures $\mu$ $\nu$, respectively. Instead,  we consider only the covariance of $\mu$, and we denote it by $Q$ without superindex.

The Sobolev spaces considered in \cite{Simone} are modeled on the classical Sobolev spaces of the Malliavin calculus, which coincide with the ones described here with the choice $R=Q^{1/2}$. 

To prove this, we first consider the Gaussian measure $\mu$. 
It is convenient to introduce an orthonormal basis of $X$ consisting of eigenvectors of $Q$, 
$Qe_k= \mu_ke_k$ for every $k\in \N$.

We recall that the Cameron-Martin space $H$ of $\mu$  coincides with  $Q^{1/2}(X)$, it is endowed with the scalar product $\langle h_1, h_2\rangle_H := \langle Q^{-1/2}h_1, Q^{-1/2}h_2\rangle$, 
and that for every $h\in H$, $h = Q^{1/2}z$,  we have 
\begin{equation}
\label{partigaussiana}
\int_X \langle Q^{1/2}\nabla \varphi(x), z\rangle  \mu(dx) = \int_X \partial_h \varphi(x) \mu(dx) = \int_X \varphi(x)  \hat{h}(x) \mu(dx), \quad \varphi\in C^1_b(X), 
\end{equation}
where $\hat{h}= R_{\mu}^{-1}h$, $R_{\mu}$ being usual extension of $Q$  to the closure of $X^*$ in $L^2(X, \mu)$. We refer to \cite{Boga} for the general theory of 
Gaussian measures in Banach spaces; all the results that we mention here about Sobolev spaces for general Gaussian measures are contained in Chapter 5 of \cite{Boga}. In our Hilbert setting 
the function $\hat{h}$ is called  white noise function $W_z$ in \cite{DP06}, and it is given by 
 \begin{equation}
\label{e7}
\hat{h} (x) = \sum_{k=1}^{\infty} \mu_k^{-1/2} \langle x, e_k\rangle \langle z, e_k\rangle , 
\end{equation}
the series being convergent in $L^p(X, \mu)$ for every $p>1$.
By definition, for every differentiable function $\varphi $ and for every $h\in H$ we have $\langle \nabla_H\varphi(x), h\rangle_H = \langle \nabla \varphi(x), h\rangle$. Therefore, $\nabla_H\varphi(x) = Q\nabla \varphi(x)$, and $|\nabla_H\varphi(x)|_H = \|Q^{1/2}\nabla \varphi(x)\|$. Our Sobolev spaces $W^{1,p}(X, \mu)$ coincide with the classical Sobolev spaces ${\mathbb D}^{1,p}(X, \nu)$ of the Malliavin calculus,  and our operator $M_p^{\mu}$ is just $Q^{-1/2}\nabla_H$. Formula \eqref{partigaussiana} is readily extended to any $\varphi\in W^{1,q}(X, \mu)$, with $q>1$. 

We recall the definition of the Gaussian divergence of $H$-valued vector fields. For a given $\Phi \in L^1(X, \mu;H)$, a function $\beta\in L^1(X, \mu)$ is called Gaussian divergence of $\Phi$, and denoted by div$_\mu \Phi$, if 
$$\int_X \langle \nabla_H \varphi, \Phi\rangle_H d\mu = -\int_X \varphi\, \beta \, d\mu, \quad \varphi\in C^1_b(X). $$
Recalling that $\nabla_H\varphi = Q\nabla \varphi$ for every $\varphi\in C^1_b(X)$ and that $\langle h,k\rangle_H = \langle Q^{-1/2}h, Q^{-1/2}k\rangle$, this means 
$$\int_X \langle Q^{1/2}\nabla \varphi, Q^{-1/2}\Phi\rangle d\mu = -\int_X \varphi\, \beta \, d\mu, \quad \varphi\in C^1_b(X). $$
So, a vector field $\Phi \in L^{p'}(X, \mu;H)$ (namely, such that such that $\widetilde{\Phi} =Q^{-1/2}\Phi \in L^{p'}(X, \mu;X)$) has Gaussian divergence div$_\mu \Phi \in L^{p'}(X, \mu)$
if and only if $\widetilde{\Phi}$ belongs to $D(M^{\mu *}_p)$, and in this case  div$_\mu \Phi = - M^{\mu *}_p \widetilde{\Phi}$.

Now, let us consider the weighted measure $\nu$. Applying \eqref{partigaussiana} to $\varphi w$, which belongs to $W^{1,q}(X, \mu)$ for every  $q>1$, we get 
$$\int_X \langle Q^{1/2}\nabla \varphi(x), z\rangle  \nu(dx) = \int_X \partial_h \varphi(x) \nu(dx) = \int_X \varphi(x)( \hat{h}(x) - \partial _h \log w(x))\nu(dx), \quad \varphi\in C^1_b(X), $$
 By the H\"older inequality,  $\hat{h}  - \partial _h \log w\in L^q(X, \nu)$ for every $q>1$, and applying once again  the H\"older inequality we obtain that Hypothesis \ref{h1'} is satisfied. Then, we   consider the Sobolev spaces $W^{1,p}(X, \nu)$ defined in the Introduction, still with $R=Q^{1/2}$. They coincide with the Sobolev spaces $W^{1,p}(X, \nu)$ of  \cite{Simone}. We remark that the test functions taken into consideration in  \cite{Simone} are the smooth cylindrical functions ${\mathcal FC}^{\infty}_b(X)$, namely functions of the type $\varphi(x) = \theta(\langle x, v_1\rangle, \ldots , \langle x, v_n\rangle)$ with $n\in \N$, $\theta  \in C^{\infty}_b(\R^n)$, $v_1$, \ldots $v_k\in X$, instead of $C^1_b(X)$ as we did. However, in the basic definitions and estimates nothing changes if we replace 
 ${\mathcal FC}^{\infty}_b(X)$ by 
$C^1_b(X)$.  

The hypersurfaces considered in  \cite{Simone} are level surfaces of functions $g$ whose regularity and summability properties are given in terms of the Gaussian measure $\mu$. Namely, as in \cite{Feyel,CeLu14}, $g\in \mathbb D^{2,p}(X, \mu)$ for every $p>1$, and there exists $\delta >0$ such that $1/|\nabla_Hg|_H \in L^p(g^{-1}(-\delta, \delta), \mu)$ for every $p>1$. Here we assume for simplicity that 
$1/|\nabla_Hg|_H \in L^p(X, \mu)$ for every $p$, which means that $1/\|M^{\mu}g\|\in  L^p(X, \mu)$ for every $p$. 
Now we prove that,  under these assumptions, $g$ satisfies Hypothesis \ref{h2}. 

\begin{Lemma}
\label{Le:reg}
Let $g\in \mathbb D^{2,p}(X, \mu)$ for every $p>1$ be such that $1/|\nabla_Hg|_H \in L^p(X, \mu)$ for every $p>1$. Then $g$ satisfies Hypothesis \ref{h2}, for both measures $\mu$ and $\nu$.
\end{Lemma}
\begin{proof} The assumption  $g\in \mathbb D^{2,p}(X, \mu)$ is equivalent to $\nabla_Hg \in \mathbb D^{1,p}(X,\mu;H)$, for every $p>1$. It follows that $\nabla_Hg/ |\nabla_Hg|_H^2 \in \mathbb D^{1,p}(X,\mu;H)$, for every $p>1$.
Every vector field  $\Phi \in \mathbb D^{1,p}(X,\mu;H)$ with $p>1$ has Gaussian divergence div$_\mu \Phi \in L^p(X, \mu)$. By the above considerations, $\Psi = Q^{-1/2}\nabla_Hg/ |\nabla_Hg|_H^2  $ belongs to the domain of $M^{\mu *}_p$, for every $p>1$. On the other hand, $Q^{-1/2}\nabla_Hg/ |\nabla_Hg|_H^2 = Mg/\|Mg\|^2$. Then, $g$ satisfies Hypothesis \ref{h2} for the measure $\mu$. 

Concerning the weighted measure  $\nu$, again we have to compare the divergence operator with our operators $M^{\nu *}_p$. 
The divergence operator is defined in  \cite{Simone} as follows, for vector fields $\Phi\in L^1(X, \nu;X)$. A function $\beta \in L^1(X, \nu)$ is called divergence of $\Phi$ and denoted by div$_{\nu}\Phi$ if 
$$\int_X \langle \nabla f(x), \Phi(x) \rangle\, \nu(dx) = -\int_X f(x)\beta(x) \nu(dx), \quad f\in C^1_b(X). $$
If $\Phi$ has values in the Cameron-Martin space $Q^{1/2}(X)$, the above formula reads as
\begin{equation}
\label{divSimone}
\int_X \langle Q^{1/2}\nabla f(x),  Q^{-1/2}\Phi(x) \rangle \, \nu(dx) = -\int_X f(x)\beta(x) \nu(dx), \quad f\in C^1_b(X). 
\end{equation}
If $\widetilde{\Phi} := Q^{-1/2}\Phi\in L^{p'}(X, \nu;X)$ and $\beta \in L^{p'}(X, \nu)$, \eqref{divSimone} means that  $ \widetilde{\Phi} \in D(M^{\nu *}_p)$,
 and $M^{\nu *}_p \widetilde{\Phi} = -\beta$. Conversely, if a vector field $\widetilde{\Phi}$ belongs to 
$D(M^{\nu *}_p)$, then $\Phi:= Q^{1/2}\widetilde{\Phi}$ has divergence in the sense of  \cite{Simone}, given by div$_{\nu}\Phi = -M^{\nu *}_p\widetilde{\Phi}$. Taking this equivalence into account, we use Proposition 5.5 of  \cite{Simone} , that states that any vector field $\Phi \in \mathbb D^{1,q}(X, \mu;H)$ has divergence div$_{\nu}\Phi$ belonging to $L^r(X, \nu)$ for every $r<q$. In our case, $\Phi = \nabla_Hg/ |\nabla_Hg|_H^2$ belongs to $\mathbb D^{1,q}(X, \mu;H)$ for every $q$, so that div$_{\nu}\Phi$ belongs to $L^q(X, \nu)$ for every $q$. Moreover, by the H\"older inequality $\widetilde{\Phi} = Q^{-1/2}\Phi$ is in $L^{p'}(X, \nu;X)$ for every $p'>1$. 
This implies that 
$\widetilde{\Phi}$ belongs to $D(M^{\nu *}_p)$ for every $p$, namely  
 Hypothesis \ref{h2} holds for the measure $\nu$. \end{proof} 

The weighted surface measure
considered in \cite{Simone} is $w^*\rho$, where $w^*$ is any $C_p$-quasicontinuous version of $w$, in the sense of the Gaussian capacity, and $\rho $ is the Gauss-Hausdorff measure of Feyel and de La Pradelle. Here we identify our surface measures $\rho_r$ with $w^*\rho$ on every surface level $g^{*-1}(r)$.

\begin{Proposition} Under the assumptions of Lemma \ref{Le:reg}, for every $r\in \R$ we have 
\begin{equation}
\label{uguali}
\int_X \varphi T(\|M^{\nu }g\|) d\sigma^g_r = \int_{g^{*-1}(r)} \varphi w^* d\rho , \quad \varphi \in C_b(X). 
\end{equation}
\end{Proposition}
\begin{proof} 
 Since any finite Borel measure is uniquely determined by its Fourier transform, it is sufficient to show that \eqref{uguali} holds for every $ \varphi \in C^1_b(X)$. 
 Theorem 
1.3 of \cite{Simone} yields, for every $\Phi \in W^{1,p}(X, \nu; H)$, 
\begin{equation}
\label{partiSimone}
\int_{\{g<r\}} {\rm div}_{\nu} (\varphi \Phi) d\nu = \int_{g^{*-1}(r)}\varphi  \langle {\rm Tr}\, \Phi,  {\rm Tr}\bigg( \frac{\nabla_Hg}{|\nabla_Hg|_H}\bigg)\rangle_H w^* d\rho   
\end{equation}
 where Tr is the  trace operator  considered in  \cite{Simone}.
There, traces Tr$\, \varphi$ of Sobolev functions $\varphi$ are defined as in the present paper, with the surface measure $w^*\rho$ replacing $\sigma^g_r$. Traces of vector fields $\Phi\in W^{1,p}(X, \nu; H)$ are defined in a natural way, namely setting $\varphi_n(x) = \langle \Phi(x), h_n\rangle_H$, where $\{h_n:\; n\in \N\}$ is any orthonormal basis of $H$, then Tr$\,\Phi = \sum_{n=1}^{\infty}$ Tr$(\, \varphi_n) h_n$. 

 Taking in particular $\Phi = \nabla_Hg /|\nabla_Hg|_H$, that  belongs to $W^{1,p}(X, \nu;H)$ for every $p>1$, 
 we have $ |$Tr $\Phi |_H^2 \equiv 1$ on $g^{*-1}(r)$, and the right hand side of \eqref{partiSimone}  is equal to 
 $$ \int_{g^{*-1}(r)}\varphi \,  w^* d\rho .   $$
Recalling that   div$_{\nu} (\varphi \Phi) = -M_p^{\nu*}(\varphi  M^{\nu}g/\| M^{\nu}g\|)$,  the left hand side is equal to 
$$- \int_{\{g<r\}}M_p^{\nu*} \bigg(\varphi  \frac{M^{\nu}g}{\| M^{\nu}g\|}\bigg) d\nu , $$
which coincides with $\int_X \varphi T(\|M^{\nu}g\|) d\sigma^g_r $ by \eqref{startingpoint2}. \end{proof}

Since the assumptions on $g$ are the same as in \cite{CeLu14,Simone}, the examples exhibited in these papers fit here.  In particular, functions such as $g(x) = \sum_{k=1}^{\infty} \alpha_k \langle x-x_0, e_k\rangle^2$ with $\alpha_k\geq 0$ for every $k$, not eventually vanishing, and $\sum_{k=1}^{\infty} \alpha_k \mu_k <\infty$, satisfy the assumptions of Lemma \ref{Le:reg}. Therefore, the theory may be applied to spherical surfaces and surfaces of suitable ellipsoids. 
The elements of the dual space $g(x) = \langle x, v\rangle$ obviously satisfy the assumptions of Lemma \ref{Le:reg}, so that the theory may be applied to hyperplanes. The hyperplane $\{x:\; \langle x, v\rangle =r\}$, with $v\in X\setminus \{0\}$,  may be seen as the graph of the function $\varphi : $ span $\{ e_k: k\neq h\}\mapsto \R$, $\varphi (\widetilde{x}) = (r - \sum_{k\neq h} \widetilde{x}_kv_k)/v_h$, if $v_h\neq 0$. A generalization to graphs of other functions is in 
 \cite{CeLu14}. 
 
 When formula \eqref{uguali} holds, Proposition \ref{q_1>0} is not needed. Since $\rho_r$  coincides with the restriction of $w^* \rho $ to $g^{*-1}(r)$, for $\rho_r$ be nontrivial it is sufficient that 
 $w^*(r)\neq 0$ and that $\rho (g^{*-1}(r))\neq 0$.  Under the assumptions of Lemma \ref{Le:reg}, the latter condition holds iff $r\in ( \essinf g, \esssup g)$ by \cite[Lemma 3.9, Prop. 3.15]{DaLuTu14}.

\section{A family of non-Gaussian product measures} 

For any $\mu>0$, $m \geq 1$,  we define the probability measure on $\R$
   \begin{equation}
\label{e1}
 \nu_{m,\mu}(d\xi):= a_{m}\,\mu^{-\frac{1}{2m}}\;e^{-\frac{|\xi|^{2m}}{2m\mu}}\,d\xi,\quad \xi\in\R,
\end{equation}
where $a_{m}$ is a normalization constant such that $ \nu_{m,\mu}(\R)=1$, 
$$
a_m=\frac{(2m)^{1-\frac{1}{2m}}}{2\,\Gamma(\frac1{2m})}.
$$
For every $N>0$ we have
\begin{equation}
\label{e2}
\int_\R | \xi|^{2N}\nu_{m,\mu}(d\xi)=a_m\mu^{-\frac{1}{2m}}\int_\R |\xi| ^{2N}e^{-\frac{|\xi|^{2m}}{2m\mu}}\,d\xi=:b_{m,N}\mu^{N/m}
\end{equation}
where
$$
b_{m,N}= a_m \int_{\R} |\tau|^{N/m} e^{-\frac{|\tau|^{2m}}{2m }}\,d\tau = (2m)^{\frac{N}{m}}\;\frac{\Gamma(\frac{2N+1}{2m})}{\Gamma(\frac1{2m})}.
$$

The measure $\nu_{m,\mu}$ has mean $0$ and covariance $b_{m,1}\mu^{\frac1m}$.
The  following integration by parts formula holds, 
  \begin{equation}
\label{e3}
\int_{\R}  \varphi'(\xi)\,\nu_{m,\mu}(d\xi)=\frac1{\mu}\int_\R |\xi|^{2m-2}\xi \varphi(\xi)\,\nu_{m,\mu}(d\xi),\quad \varphi\in C^1_b(\R).
\end{equation}

Next,  we  define a product measure on $\R^\infty$, the space of all sequence of real numbers endowed with the product topology, associated to the distance
$d(x, y) = \sum_{n=1}^{\infty} 2^{-n}|x_n-y_n|(1+ |x_n-y_n|)^{-1}$. We set 

\begin{equation}
\label{e4}
\nu_m=\prod_{h=1}^\infty \nu_{m,\mu_h},
\end{equation}
where the sequence of positive numbers $(\mu_h)$ is chosen such as
\begin{equation}
\label{e5}
\Lambda_m := \sum_{h=1}^\infty\mu_h^{\frac1m}<\infty.
\end{equation}

 As usual, we denote by  $\ell^2$ the space of all sequence  $(x_h)$ of real numbers such that $\sum_{h=1}^\infty x_h^2<\infty$, 
 endowed with the scalar product
$$
\langle x, y \rangle=\sum_{h=1}^\infty x_hy_h,\quad x,y\in \ell^2.
$$

One checks easily that  $\ell^2$ is a Borel set in $\R^\infty$ and that $\nu$ is concentrated on $\ell^2$ because, in view of \eqref{e2}
$$\int_{\R^{\infty}}  |x|^2_{\ell_2}\,\nu(dx)=\sum_{h=1}^\infty\int_\R x_h^2\nu_{m,\mu_h}(dx_h)=b_{m,1}\sum_{k=1}^\infty\mu_h^{\frac1m}<\infty.$$
So, from now on we may forget $\R^\infty$ and consider only $\ell^2$, identifying it with $X$ through the mapping $x\mapsto (x_h)$, where $x_h = \langle x, e_h\rangle$ and  
$\{ e_h:\, h\in \N\} $ is any fixed  orthonormal basis of $X$.

One check easily that $\nu$ has mean $0$ and that it possesses   finite   moments of any order. The covariance $Q$ of $\nu $ is given by
\begin{equation}
\label{e29h}
Qe_h= b_{m,1}\,\mu_h^{\frac1m}\,e_h,\quad h\in\N. 
\end{equation}
Notice that if $m=1$ then  $\nu_1$ is  the Gaussian measure   $N_{0, Q}$. In this case  $Qe_h=  \mu_h \,e_h$, for all $ h\in\N$,  and
for all  $\varphi\in C^1_b(X)$, $z\in Q^{1/2}(X)$ the   classical integration  formula \eqref{partigaussiana} holds. 
%
%
%

We are going to generalize formula \eqref{partigaussiana}  to any $\nu_m$ with  $m\geq 1$.

\begin{Proposition}
\label{p13}
Let $m\geq 1$,  $\varphi\in C^1_b(X)$, $z\in X$. Then  
\begin{equation}
\label{e10}
\int_X \langle  Q^{\frac1{2}}\nabla \varphi(x),z\rangle \,\nu_{m}(dx)=\int_X W^m_z(x)\,\varphi(x)\,\nu_{m}(dx),
\end{equation}
where
\begin{equation}
\label{e9}
W^m_z(x)  := b^{1/2}_{m,1}\sum_{h=1}^\infty  \mu_h^{\frac1{2m}-1}\, |x_h|^{2m-2}x_hz_h,
\end{equation}
the series being convergent in  $L^p(X,\nu_m)$ for every $p\in(1, +\infty)$.  Consequently, Hypothesis  \ref{h1'} is satisfied, with $R=Q^{1/2}$ and $C_{p,z} = \|W^m_z\|_{L^{p'}(X, \nu)}$.
\end{Proposition}
\begin{proof}
As a first step, we prove that for every $\varphi\in C^1_b(X)$, $h\in \N$ we have
\begin{equation}
\label{partih}
\int_X \frac{\partial  \varphi}{\partial e_h}(x)\,  \nu_m(dx) =  \mu_h^{\frac1{2m}-1} \int_X|x_h|^{2m-2}x_h \varphi (x) \,\nu_m(dx). 
\end{equation}
To this aim we approach $\varphi$ by a sequence of cylindrical functions, $\varphi_n(x) : = \varphi(P_n x)$, where $P_n$ is the orthogonal projection 
$$P_n(x) = \sum_{k=1}^n \langle x, e_k \rangle e_k. $$
The sequence $(\varphi_n)$ converges to $\varphi $ in $W^{1,p}(X, \nu_m)$ for every $p\in (1, +\infty)$. Indeed, it converges in $L^{p}(X, \nu_m)$ by the Dominated Convergence Theorem, and  moreover
$$Q^{1/2}\nabla \varphi_n(x)= Q^{1/2}P_n \nabla\varphi(P_nx), \quad n\in \N, $$
so that 
$$\begin{array}{l}
  \| Q^{1/2}\nabla \varphi_n  - Q^{1/2}\nabla \varphi \|_{L^p(X, \nu_m;X)} \leq
\\
\\
\ds \leq \bigg( \int_X \|Q^{1/2 }( P_n\nabla\varphi(P_nx) - P_n \nabla\varphi(x))\|^p\,\nu_m(dx) \bigg)^{1/p}
+ \bigg( \int_X \|Q^{1/2 }( P_n\nabla\varphi( x) -   \nabla\varphi(x))\|^p\,\nu_m(dx) \bigg)^{1/p}
\\
\\
\ds \leq  \|Q^{1/2 }\|_{{\mathcal L}(X)}  \bigg( \int_X \|\nabla\varphi(P_nx) - \nabla\varphi(x)\|^p \bigg)^{1/p}
+ \|Q^{1/2 }\|_{{\mathcal L}(X)}  \bigg( \int_X \| P_n\nabla\varphi( x) -   \nabla\varphi(x)\|^p\,\nu_m(dx) \bigg)^{1/p}
\end{array}$$
where both integrals in the right hand side vanish as $n\to \infty$ by the Dominated Convergence Theorem. 

So, it is enough to prove that \eqref{partih} holds for cylindrical functions of the type $\varphi(x) = \widetilde{\varphi}(x_1, \ldots, x_n)$ for some $ \widetilde{\varphi}\in C^1_b(\R^n)$, $n\in \N$. For such functions, 
$$\int_X \frac{\partial  \varphi}{\partial e_h}(x)\,  \nu_m(dx) = \int_{\R^n} \frac{\partial \widetilde{\varphi}}{\partial \xi_h} \Pi_{k=1}^n \nu_{m, \mu_k}(d\xi)$$
 and 
  \eqref{partih} is an immediate consequence of \eqref{e3}. 

Let now  $\varphi\in C^1_b(X)$, $z\in X$. We have
$$\begin{array}{lll}
\ds \int_X \langle  Q^{\frac1{2}}\nabla \varphi(x),z\rangle \,\nu_{m}(dx) & = & \ds \lim_{n\to \infty} \int_X \sum_{h=1}^n b_{m,1}^{1/2}\mu_h^{1/2m} \frac{\partial  \varphi}{\partial e_h}(x)z_h\, \nu_m(dx) 
\\
\\
& = & \ds
 \lim_{n\to \infty} b_{m,1}^{1/2} \int_X \sum_{h=1}^n  \mu_h^{\frac1{2m}-1}|x_h|^{2m-2}x_h \varphi (x) \,z_h \,\nu_m(dx). 
 \end{array}$$
To conclude the proof it is enough to show that the series 
$$
s_n(x) := \sum_{h=1}^n  \mu_h^{\frac1{2m}-1}\, |x_h^{2m-2}|x_h\,z_h
$$
is convergent in $L^{2p}(X,\nu_m)$ for every $p\in \N$. Recalling that 
$$(a_1+\ldots +a_n)^{2p} =   \sum_{k_1, \ldots, k_n \in \{0, \ldots, 2p\}, \; \sum_{j=1}^n k_j = 2p} \frac{(2p)!}{(k_1)!\cdot \ldots \cdot (k_n)!} a_1^{k_1}\cdot\ldots \cdot a_n^{k_n}$$
for every $l$, $n\in \N$ we get
 $$
\begin{array}{c}
 (s_{l+n}(x) - s_l(x))^{2p} =   
  \\
 \\
 \ds  =  (2p)! \sum_{k_1, \ldots, k_n \in \{0, \ldots, 2p\}, \; \sum_{j=1}^n k_j = 2p} \;  \prod_{j=1}^n \frac{1}{(k_j)! }\mu_{l+j}^{(\frac{1}{2m}-1)k_j}\, |x_{l+j}|^{(2m-2)k_j}
(x_{l+j}z_{l+j})^{k_j}. \end{array}$$
 Integrating with respect to $\nu_m$, the integrals of the terms with some odd $k_j$ vanish. What remains are the integrals of the terms where all the $k_j =2h_j$ are even, and recalling that 
 $$\int_X  x_{l+1}^{2(2m-1)h_1}\cdot \ldots \cdot  x_{l+n}^{2(2m-1)h_n}\nu_m(dx) =  \prod  _{j=1}^n b_{m, (2m-1)h_j}\mu_{l+j}^{(2-1/m)h_j } $$ 
 we get 
$$\begin{array}{l}
\ds \int_X   (s_{l+n}n(x) - s_l(x))^{2p}\nu_m(dx) = 
\\
\\
\ds  \sum_{h_1, \ldots, h_n \in \{0, \ldots, p\}, \; \sum_{j=1}^n h_j = p} \frac{(2p)!}{(2h_1)!\cdot \ldots \cdot (2h_n)!} \int_X \prod_{j=1}^n  \mu_{l+j}^{(\frac{1}{2m}-1)2h_j}\, |x_{l+j}|^{2(2m-1)h_j}
 z_{l+j}^{2h_j} \,d\nu_m
 \\
 \\
 = \ds   \sum_{h_1, \ldots, h_n \in \{0, \ldots, p\}, \; \sum_{j=1}^n h_j = p} \frac{(2p)!}{(2h_1)!\cdot \ldots \cdot (2h_n)!}  \prod  _{j=1}^n b_{m, (2m-1)h_j} z_{l+j}^{2h_j}  
 \\
 \\
\ds  \leq c_{m,p} \bigg(\sum_{j=1}^n  z_{l+j}^2\bigg)^p , 
 \end{array}$$
where $c_{m,p}= (\max \{ b_{m, (2m-1)h}:\; h=0, \ldots, p\})^p$. So, $(s_n)$ is a Cauchy series in $L^{2p}(X,\nu_m)$. \end{proof}

Proposition \ref{p13} yields   the following corollary.

\begin{Corollary}
Let $m\in\N$, and let \eqref{e5} hold. For every  $\varphi,\psi\in C^1_b(X)$, $z\in X$ we have
\begin{equation}
\label{e11}
\begin{array}{lll}
\ds \int_X \langle  Q^{\frac1{2}}\nabla \varphi(x),z\rangle \,\psi(x)\,\nu_{m}(dx)&=&\ds-\int_X \langle  Q^{\frac1{2}}\nabla \psi(x),z\rangle \,\varphi(x)\,\nu_{m}(dx)\\
\\
\ds&& + \ds\int_X W^m_z(x)\,\varphi(x)\,\psi(x)\,\nu_{m}(dx).
\end{array}
\end{equation}
In particular, 
$$\bigg| \int_X \langle  Q^{\frac1{2}}\nabla \varphi(x),z\rangle )\,\nu_{m}(dx)\bigg| \leq \|\varphi\|_{L^p(X, \nu_m)} \|W^m_z\|_{L^{p'}(X, \nu_m)}. $$
\end{Corollary}

 



 Consequently, Hypothesis  \ref{h1} is satisfied, and all the results of Section 2 hold. 

  According to the notation of Section 1,  we denote by $M_p$ the closure of   $ Q^{1/2}\nabla :C^1_b(X) \mapsto L^p(X, \nu_m;X)$ in $L^p(X, \nu_m)$ and by $W^{1,p}(X,\nu_m)$ the domain of $M_p$.

We shall show that our surface measures are well defined on hyperplanes and spherical surfaces. For simplicity, we consider only balls centered at the origin.

 \subsubsection{Spherical surfaces}

 Here we take  $g(x)=\|x\|^2$, $x\in X$. Then $g$ is smooth 
 and $\{g <r\}$ is the open ball of center $0$ and radius $\sqrt r$, for $r\geq 0$. In this case the vector field $Mg/\|Mg\|^2$ in Hypothesis \ref{h2} is given by 
 $$
 \Psi (x) =\frac{Q^{1/2}x}{2\|Q^{1/2}x\|^2}. 
 $$

We have to prove that  $ \Psi \in D(M_p^*)$ for every $p>1$. We approach it by the sequence of vector fields 
 $S_n(x) = \sum_{h=1}^n \langle \Psi(x), e_h\rangle e_h$ that are sums of vector fields  of the type considered in Lemma \ref{Rem:divergenza}, with 
 $$f_h (x)=  \langle \Psi(x), e_h\rangle = b_{m,1}^{1/2}\mu_h^{1/2m}x_k /2\|Q^{1/2}x\|^2. $$
 
We use the following lemma. 

\begin{Lemma}
\label{Le:derivate}
\begin{itemize}
\item[(i)] The function $x\mapsto \|Q^{1/2}x\|^{-1}$ belongs to $L^q(X, \nu_m)$ for every $q>1$. 
\item[(ii)] For every $k\in \N$, the function $\varphi_k(x) := x_k/ \|Q^{1/2}x\|^2$ belongs to $W^{1,q}(X, \nu_m)$ for every $q>1$, and 
\begin{equation}
\label{Mvarphik}
M_q\varphi_k = \sum_{h=1}^{\infty} b_{m,1}^{1/2}\mu_h^{1/2m} \bigg( \frac{\delta_{h,k}}{\|Q^{1/2}x\|^2 } - \frac{b_{m,1} \mu_h^{1/m} x_hx_k}{(\|Q^{1/2}x\|^2  )^2}\bigg)e_h . 
\end{equation}
\end{itemize}
\end{Lemma} 
\begin{proof}
The proof of statement (i) is the same as in the Gaussian case $m=1$; we write it for the reader's convenience. Let $p>1$. Since $1/\|Q^{1/2}x\| \leq 1/\|P_n Q^{1/2}x\|$ for every $n\in \N$ and $x\in X \setminus \{0\}$, it is sufficient to show that $x\mapsto 1/\|P_n Q^{1/2}x\|\in L^p (X, \nu_m)$ for a suitable $n$. For every $x\in X$ we have 
$$\|P_n Q^{1/2}x\|^2 = b_{m,1}\sum_{h=1}^n x_h^2 \mu_h^{1/m} \geq b_{m,1}(\min\{ \mu_h: \; h=1, \ldots , n\} )^{1/m} \sum_{h=1}^n x_h^2, $$
so that 
$$\int_X \frac{1}{\|P_n Q^{1/2}x\|^p}\, d\nu_m \leq \frac{1}{ (b_{m,1}(\min\{ \mu_h: \; h=1, \ldots , n\} )^{1/m})^{p/2}}\int_{\R^n}  \bigg( \sum_{h=1}^n \xi_h^2\bigg)^{-p/2}\Pi_{h=1}^n \nu_{m, \mu_h}(d\xi) $$
which is finite for $n>p$. 

Let us prove statement (ii). We approach  $\varphi_k$   by the functions
$$\varphi_{k,n}(x) = \frac{x_k}{\|Q^{1/2}x\|^2 + 1/n}, $$
that belong to $C^1_b(X)$ and that are easily seen to converge to $\varphi_k$ in $L^q(X, \nu_m)$ for every $q>1$, taking  (i) into account. Moreover we have 
$$\langle Q^{1/2}\nabla \varphi_{k,n}(x) , e_h\rangle = b_{m,1}^{1/2}\mu_h^{1/2m} \bigg( \frac{\delta_{h,k}}{\|Q^{1/2}x\|^2 + 1/n} - \frac{b_{m,1} \mu_h^{1/m} x_hx_k}{(\|Q^{1/2}x\|^2 + 1/n)^2}\bigg), \quad h\in \N. $$
Denoting by $F$ the vector field in the right-hand side of \eqref{Mvarphik} and using again (i), we see that $\lim_{n\to \infty}$ $\|Q^{1/2}\nabla \varphi_{k,n}-F\| =0$ in $L^q(X, \nu_m)$ for every $q>1$. Statement (ii) follows.  
 \end{proof}

\begin{Proposition}
The function $g(x) = \|x\|^2$ satisfies Hypothesis \ref{h2}, and $Mg \in W^{1,q}(X, \nu_m;X)$ for every $q>1$. 
\end{Proposition}
\begin{proof}
 By Lemma \ref{Le:derivate} and Lemma \ref{Rem:divergenza}, for every $k\in \N$ the vector field $f_k(x) e_k$ belongs to $D(M_p^*)$ for every $p>1$, and by \eqref{formulaM^*_p}   we have
 $$M_p^* (f_ke_k) = - \frac{b_{m,1} \mu_k^{1/m}}{2} \bigg(  \frac{1}{\|Q^{1/2}x\|^2} - 2b_{m,1} \frac{\mu_k^{1/m}x_k^2}{\|Q^{1/2}x\|^4}\bigg) + \frac{b_{m,1}  }{2} \frac{ \mu_k^{1/m -1}|x_k|^{2m}}{\|Q^{1/2}x\|^2}. $$
Therefore, the series $S_n(x) = \sum_{h=1}^n f_k (x) e_k$ converges pointwise to 
\begin{equation}
\label{eq:divPsi0}
 \frac{1}{2} \bigg( -\frac{ \mbox{\rm Tr}\;Q}{\|Q^{1/2}x\|^2}  + \frac{2\|Q^2x\|^2}{\|Q^{1/2}x\|^4}\bigg) + \frac{b_{m,1}}{2\|Q^{1/2}x\|^2} \sum_{k=1}^{\infty}  \mu_k^{1/m -1} |x_k|^{2m}  
 \end{equation}
where the series $ \sum_{k=1}^{\infty}  \mu_k^{1/m -1} x_k^{2m}$ converges in $L^q(X, \nu_m)$ for every $q>1$, since $(\int_X |x_k|^{2mq}\nu_m(dx))^{1/q}$ $ = b_{m,mq}^{1/q}\mu_k $. By Lemma 
 \ref{Le:derivate}(i), $x\mapsto 1/\|Q^{1/2}x\|^2 \in L^s(X, \nu_m)$ for every $s>1$. Therefore, $(S_n)$ converges to the right hand side of \eqref{eq:divPsi0} in $L^p(X, \nu_m)$ for every $p>1$. So, 
 $\Psi\in D(M_p^*)$ and 
 \begin{equation}
\label{eq:divPsi}
M_p^*\Psi=  \frac{1}{2} \bigg( -\frac{ \mbox{\rm Tr}\;Q}{\|Q^{1/2}x\|^2}  + \frac{2\|Q^2x\|^2}{\|Q^{1/2}x\|^4}\bigg) + \frac{b_{m,1}}{2\|Q^{1/2}x\|^2} \sum_{k=1}^{\infty}  \mu_k^{1/m -1} x_k^{2m}  . 
 \end{equation}
Hypothesis \ref{h2} is so fulfilled. Moreover, the vector field $Mg(x) = 2 Q^{1/2}x$ belongs to $W^{1,q}(X, \nu_m;X)$ for every $q>1$, since every component $f_i(x) = 2b_{m,1}^{1/2}\mu_i^{1/2m}x_i$  is in $W^{1,p}(X, \nu_m)$, and $\sum_{i=1}^{\infty} \|M_qf_i(x)\|^2 = 4 b_{m,1}^2\sum_{i=1}^{\infty}  \mu_i ^{1/m}$ is a real constant by assumption \eqref{e5}. Therefore, the assumptions of Proposition \ref{Th:perimeter} are satisfied. 
\end{proof}

For every $r>0$, let $\sigma^g_r$ be the measure given by Theorem \ref{costruzione}. Setting 
$$\rho_r(dx) : = 2 \| Q^{1/2}x\| \sigma^g_r(dx), $$
 formula \eqref{Th:divergenza} reads as
$$\int_{B(0,r)} \langle M_p\varphi, F\rangle \, d\nu_m = \int_{B(0,r)} \varphi M^*_pF\, d\nu_m  + \int_{\partial B(0,r)} T\bigg(\varphi  \langle F(x), \frac{Q^{1/2}x}{\|Q^{1/2}x\|}\rangle  \bigg)\rho_r(dx) ,  $$
 for every $F\in D(M^*_p)$,  $ \varphi\in W^{1,q}(X, \nu_m)$ with $q>p$. In particular, for a constant vector field $F(x)\equiv z$ and $\varphi \in C^1(X)\cap W^{1,q}(X, \nu_m)$ for some $q$ we get 
$$\int_{B(0,r)} \langle Q^{1/2}\nabla \varphi, z\rangle \, d\nu_m = \int_{B(0,r)} \varphi W_z^m\, d\nu_m  + \int_{\partial B(0,r)} \varphi  \langle z, \frac{Q^{1/2}x}{\|Q^{1/2}x\|} \rangle\, \rho_r(dx). $$

  \subsubsection{Hyperplanes}
 
 We take here $g(x)=\langle x,a \rangle$ where $a\in X\setminus \{0\}$ is fixed. Then
 $$\nabla g(x)=a, \quad x\in X,   $$
and the vector field $\Psi(x) = Mg(x)/\|Mg(x)\|^2$ of Hypothesis \ref{h2} is constant, equal to 
 $$\Psi(x) =\frac{Q^{1/2}a}{\|Q^{1/2}a\|^2}, \quad x\in X. $$
By Proposition \ref{p13},  Hypothesis \ref{h1'} is satisfied,  and therefore $\Psi \in D(M_p^*)$ for every $p\in (1, +\infty)$. By 
  \eqref{e9} it follows that
 \begin{equation}
\label{e122}
M_p^*(\Psi)(x)= \frac{v_{Q^{1/2}a}(x)}{\|Q^{1/2}a\|^2} =
\frac{b_{m,1}}{\|Q^{1/2}a\|^2} \sum_{h=1}^\infty \mu_h^{-1+1/m}|x_h|^{2m-2}x_h a_h. 
\end{equation}
Therefore, $g$ satisfies Hypothesis \ref{h2}. Since $Mg$ is constant, it belongs to all $W^{1,q}(X, \nu_m)$ spaces, and also the hypotheses of Proposition \ref{Th:perimeter} are satisfied.The normalized surface measure $\rho_r$  on the hyperplane $\{x: \; \langle x,a \rangle =r\}$  is now 
$$\rho_r(dx) =  \| Q^{1/2}a\| \sigma^g_r(dx), $$
for every $r \in \R$, where $\sigma^g_r$ is the measure given by Theorem \ref{costruzione}.  Formula \eqref{Th:divergenza} reads as
$$\int_{\{x:\, \langle x,a \rangle <r\} } \langle M_p\varphi, F\rangle \, d\nu_m = \int_{\{x:\, \langle x,a \rangle <r\} } \varphi M^*_pF\, d\nu_m  + \int_{\{x:\, \langle x,a \rangle =r\} } T\bigg(\varphi  \langle F(x), \frac{Q^{1/2}a}{\|Q^{1/2}a\| }\rangle \bigg)\, \rho_r(dx) ,  $$
 for every $F\in D(M^*_p)$,  $ \varphi\in W^{1,q}(X, \nu_m)$ with $q>p$. In particular, for a constant vector field $F(x)\equiv z$ and $\varphi \in C^1(X)\cap W^{1,q}(X, \nu_m)$ for some $q$ we get 
$$\int_{\{x:\, \langle x,a \rangle <r\}} \langle Q^{1/2}\nabla \varphi, z\rangle \, d\nu_m = \int_{\{x:\, \langle x,a \rangle <r\}} \varphi W_z^m\, d\nu_m  + \langle z, \frac{Q^{1/2}a}{\|Q^{1/2}a\| }\rangle\int_{\{x:\, \langle x,a \rangle =r\}} \varphi   \rho_r(dx). $$

 \subsection{A Markov semigroup having $\nu$ as an invariant measure}

We are going to construct a transition semigroup $P_t,\,t\ge 0,$ on  $X$ that has $\nu_m$ as an invariant measure.

To this purpose we start by introducing a family of  ordinary stochastic  differential equations,  indexed by $h\in \N$, 
 \begin{equation}
\label{e43}
\left\{\begin{array}{l}
dX_h=-\frac{1}{2\mu_h}  |X_h|^{2m-2}X_hdt+dW_h(t),\\
\\
 X_h(0)=x_h\in \R, 
\end{array}\right. 
\end{equation}
where $(W_h)$ is a sequence of real  mutually independent Brownian motions defined in a probability space $(\Omega,\mathcal F,\P). $

For any $h\in \N$ equation \eqref{e43} has a unique solution $X_h(t,x_h)$.  So, we can introduce a family of  transition semigroups on $\R$, 
  \begin{equation}
\label{e43c}
(P_t^h f)(\xi)=\E[f(X_h(t,\xi))],\quad f\in C_b(\R).
\end{equation}
Moreover, the measure   $ \nu_{m,\mu_h}$ (see \eqref{e1}) is the unique invariant measure of $P_t^h$, namely it is the unique Borel probability measure $\nu$ in $\R$ such that 
 \begin{equation}
\label{e44c}
\int_\R (P_t^h f)(\xi)\nu (d\xi)=\int_\R   f(\xi)\nu (d\xi),\quad \forall f\in C_b(\R), \; t\geq 0.
\end{equation}
See e.g. \cite[Ch. 2]{Ce01}, or else \cite[Ch. 4]{Kh12}. 

Similarly for any $N\in\N$ we introduce a transition semigroup in $\R^N$ setting
 \begin{equation}
\label{e43cc}
(P_t^{(N)} \varphi) (\xi)=\E[\varphi(X_1(t,\xi_1),\cdots,X_N(t,\xi_n))],\quad \varphi\in C_b(\R^N), 
\end{equation}
and $\prod_{h=1}^N\nu_{m,\mu_h}$ is an invariant measure for  $P_t^{(N)} $, so that 
 \begin{equation}
\label{e44cc}
\int_\R (P_t^{(N)} \varphi) (x)\prod_{h=1}^N\nu_{m,\mu_h}(dx)=\int_{\R^N}   \varphi(x)\prod_{h=1}^N\nu_{m,\mu_h}(dx),\quad\forall \varphi\in C_b(\R^N), \, t\geq 0.
\end{equation}

We are now ready to show  the main  result of this section. We fix an orthonormal basis  $\{e_h:\, h\in \N \}$ of $X$, and for every $x\in X$ and $h\in \N$  we set as usual $x_h:= \langle x, e_h\rangle$. 

\begin{Proposition}
\label{p14}
For any $x\in X$ define 
\begin{equation}
\label{e44b}
X(t,x):= \sum_{h=1}^\infty   X_h(t,x_h)e_h ,
\end{equation}
Then $X(t,x),\;t\ge 0$,  is a stochastic process in $X$. 
Moreover, defining the corresponding transition semigroup by 
\begin{equation}
\label{e44bb}
P_t\varphi(x):=\E[\varphi(X(t,x))], \quad  \varphi\in C_b(X),
\end{equation} 
$\nu_m$ is an invariant measure of $P_t $.
\end{Proposition}
\begin{proof}

First we   show that for all $x\in X$ we have
\begin{equation}
\label{e43b}
\sum_{h=1}^\infty\E|X_h(t,x_h)|^2<\infty,\quad\forall\;t>0.
\end{equation}
Then
 we   define the process
 \begin{equation}
\label{e44d}
X(t,x):= \sum_{h=1}^\infty   X_h(t,x_h)e_h ,
\end{equation}
and prove  that $\nu_m$ is invariant for   $P_t $. We proceed in three steps.\medskip

{\em Step 1: a preliminary estimate.}  Let $v$ be the solution of the initial value problem
 \begin{equation}
  \label{e47b}
  \left\{\begin{array}{l}
  v'=1-\alpha v^m,\\
  \\
   v(0)=\xi\ge 0,
  \end{array}\right.
  \end{equation}
  where $\alpha>0$ and $m\in \N$.
We shall show that the following estimate holds, 
\begin{equation}
\label{e48b}
v(t) \le \max \{  \alpha^{-1/m} , v(0)\}, \quad t\geq 0.
\end{equation}

First we notice that the only stationary point  of the equation is $c:= \alpha^{-1/m}$.  
Consequently if
 $v(0)\le c$ we have $v(t)\le c$ and if $v(0)\ge c$ we have $v(t)\ge c$, as long as $v(t)$ exists. Moreover, in the first case $v$ is increasing and in the second case it decreases. This implies that $v$ is bounded and therefore it is defined in $[0, +\infty)$ and $0\leq v(t) \leq c$ in the first case, $c\leq v(t) \leq v(0)$ in the second case. 
  \eqref{e48b} follows. \medskip

{\em Step 2:  Proof of \eqref{e43b}.}\medskip

From It\^o's formula we get
\begin{equation}
\label{e46}
\frac{d}{dt}\;\E|X_h(t,x)|^2=-\mu_h^{-1}\E|X_h(t,x)|^{2m}+1.
\end{equation}
 Since
$$
(\E|X_h(t,x)|^{2})^m\le \E|X_h(t,x)|^{2m},
$$
we get
$$
\frac{d}{dt}\;\E|X_h(t,x)|^2\le -\mu_h^{-1}(\E|X_h(t,x)|^{2})^m+1.
$$
 
A standard comparison result yields
$$
\E|X_h(t,x)|^2\le u_h(t),\quad t\ge 0,
$$
where $u_h$ is the nonnegative solution of the initial value problem
\begin{equation}
\label{e44a}
\left\{\begin{array}{l}
u_h'(t)=1- \mu_h^{-1}u_h^m, \\
\\
u_h(0)=x_h^2. 
\end{array}\right.
\end{equation}
By Step 1  it follows that
\begin{equation}
\label{e51b}
 \E\|X(t,x)\|^2\le \sum_{h=1}^\infty u_h(t)\le\sum_{h=1}^\infty
\max\{\mu_h^{1/m}, x^2_h \}, \quad t\geq 0.
\end{equation}
In particular, for all  $t\ge 0$,   \eqref{e43b} is fulfilled and  we have
\begin{equation}
\label{e52b}
 \E\|X(t,x)\|^2  \le   \Lambda_m + \|x\|^2,
\end{equation}
where $\Lambda_m$ is defined in  \eqref{e5}. So, \eqref{e43b} is proved.\medskip

{\em Step 3: $\nu_m$ is invariant for $P_t$.}

\medskip We have to show that 
\begin{equation}
\label{e55g}
\int_XP_t\varphi(x)\nu_m(dx)=\int_X\varphi(x)\nu_m(dx),\quad\forall\;\varphi\in C_b(X).
\end{equation} 
\medskip

 Equation \eqref{e55g} is  fulfilled if $\varphi$ is cylindrical, by \eqref{e44cc}. The conclusion follows by approximating pointwise any
 function $\varphi\in C_b(X)$ by cylindrical functions and taking into account  \eqref{e43b}. 
 \end{proof}

 
  \section{Some   invariant measures of SPDEs}

Here we   consider the invariant measures  of a stochastic reaction--diffusion equation (Section \ref{Reaction--Diffusion}) and of  the stochastic Burgers equation (Section \ref{Burgers}) in the space $X=L^2(0,1)$. We shall show that  surface integrals can be defined in both  cases on smooth surfaces such as spherical surfaces and hyperplanes of $X $.

Such equations look like
\begin{equation}
\label{e6.1}
\left\{\begin{array}{lll}
dX(t)=[AX(t)+f(X(t))]dt+(-A)^{-\gamma/2}dW(t),\\
\\
X(0)=x.
\end{array}\right.
\end{equation}
with $\gamma \in [0, 1)$. 
In both cases, $A$ is  the realization of the second order derivative in $X=L^2(0,1)$ with
Dirichlet boundary conditions,
$$D(A)=H^2(0,1)\cap H^1_0(0,1), \quad Ax(\xi) =  x''(\xi),  $$
$W$ is  an $X$--valued cylindrical Wiener process, and $f$ is a suitable function: either it is the composition with a polynomial, $f(x)(\xi) = \sum_{k=0}^d a_k (x(\xi))^k$, or $f(x)(\xi) = x(\xi)x'(\xi)$ for $x\in H^1(0,1)$, $\xi\in (0, 1)$.

We   consider the complete orthonormal system in $X$   given by 
$$\{e_h(\xi) := \sqrt{2}\sin(h\pi\xi), \quad  h\in \N\}, $$
  consisting  of eigenfunctions of $A$, 
since  
$$Ae_h=-h^2\pi^2 e_h =: -\alpha_he_h, \quad h\in \N. $$
 We recall that  $D((-A)^{\beta})=H^{2\beta}(0,1) \cap H^1_0(0,1)$ for all 
$\beta \in ( 1/2, 1]$. 

As in the previous section we set 
$$x_h := \langle x, e_h\rangle , \quad x\in X, \; h\in \N, $$
and for every $n\in \N$ we denote by $P_n$ the orthogonal projection on the subspace generated by $e_1$, \ldots , $e_n$,  namely
\begin{equation}
\label{Pn}
 P_nx :=   \sum_{h=1}^n x_h e_h. 
 \end{equation}

Moreover, we consider the space ${\mathcal E}_A (X)$, consisting of the linear span of real and imaginary parts of the functions $x\mapsto e^{i\langle x, y\rangle}$ with $y\in D(A)$.

The following approximation lemma will be used in both examples. 

\begin{Lemma}
\label{Le:Fejer}
Let $h\in \N \cup \{0\}$. For every $\varphi \in C^h_b(\R^n)$  there exists a sequence of trigonometric polynomials $\varphi_k$ (namely, functions in the  linear span of real and imaginary parts of the functions $x\mapsto \exp(i\langle x, a\rangle _{\R^n})$, with $a\in \R^n$) such that for every multi-index $\alpha$ with $0\leq |\alpha| \leq h$ we have
\begin{itemize}
\item[(i)] $\lim_{k\to \infty} D^{\alpha }\varphi_k(x) =  D^{\alpha }\varphi(x)$,  for every $ x\in \R^n$, 
\item[(ii)] $\| D^{\alpha }\varphi_k\|_{\infty} \leq C  \| D^{\alpha }\varphi \|_{\infty}$, 
\end{itemize}
where the constant $C$ depends only on $h$ and $n$. 
\end{Lemma}
\begin{proof}
The result is classical for functions that are periodic in each variable.  Indeed, if $\varphi $ is $1$-periodic in all the variables we can take the convolutions with the Fejer kernels, 
$$\varphi_N(x) = \int_{[-1/2, 1/2]^n} K_N(y)\varphi(x-y)dy, \quad N\in \N, $$
with 
$$K_N(y) = \prod_{j=1}^n \frac{1}{N+1} \bigg( \frac{\sin \pi(N+1)y_j}{\sin \pi y_j}\bigg)^2, \quad N\in \N. $$
Then, $\|K_N\|_{L^1([-1/2, 1/2]^n)} =1$ for every $N$, and $ D^{\alpha} K_N \ast \varphi = K_N \ast  D^{\alpha}   \varphi $ converges uniformly to $  D^{\alpha}   \varphi $, for $|\alpha| \leq h$. In this case, the constant $C $ is $1$. See e.g. \cite[Exercise 73]{DS58}, or \cite{So84} for detailed proofs. 
 
If $\varphi$ is $T$-periodic in all variables, the convolutions over $[-T/2, T/2]^n$ with the rescaled Fejer kernels $K_{N,T}(y) : = K_N(y/T)/T^n$ make the same job. It is important to notice that the constant $C $ is still $1$. 

If  $\varphi$ is not periodic, we consider a sequence $\widetilde{\varphi}_k$ of functions that are $k$-periodic in all variables, and coincide with $\varphi $ in $[(-k+1)/2, (k-1)/2]^n$. To construct such a sequence, we take $\theta_k \in C^{\infty}(\R)$ such that $\theta_k \equiv 1$ in $[(-k+1)/2, (k-1)/2]$, $\theta_k \equiv 0 $ outside $[-k/2, k/2]$ and $\|\theta_k\|_{C^h(\R)}$ bounded by a constant  independent of $k$, and we define $\widetilde{\varphi}_k$ as the $k$-periodic function in all variables, that coincides with $\varphi(x)\prod_{j_1}^d\theta(x_j)$ in $[-k/2, k/2]^n$. 
So, there are constants $C_{|\alpha|}$, independent of $k$ and $\varphi$, such that $  \| D^{\alpha}\widetilde{\varphi}_k\|_{\infty} \leq C_{|\alpha|} \| D^{\alpha }\varphi \|_{\infty}$, for $0\leq |\alpha| \leq h$. 

By the above procedure, for every $k$ there exists a trigonometric polynomial $\varphi_k$ such that $\| D^{\alpha}(\varphi_k - \widetilde{\varphi}_k )\|_{\infty} \leq 1/k$, and 
$\| D^{\alpha}\varphi_k\|_{\infty} \leq  \| D^{\alpha}\widetilde{\varphi}_k\|_{\infty} \leq C_{|\alpha|} \| D^{\alpha }\varphi \|_{\infty}$, for  $0\leq |\alpha| \leq h$, so that (ii) holds. Since $\widetilde{\varphi}_k$ coincides with $\varphi $ in $[(-k+1)/2, (k-1)/2]^n$, the sequence 
$(\varphi_k)$  satisfies also (i). 
\end{proof}

\subsection{Reaction--Diffusion equations}
\label{Reaction--Diffusion}

Here we consider problem \eqref{e6.1} where $f(x)$ is the composition of  a decreasing polynomial  of odd degree $d$ greater than $1$ with $x$, 
$$f(x)(\xi )= \sum_{k=1}^d a_k (x(\xi))^k, \quad x\in X, \; \xi\in (0, 1). $$
 
 It is well known that for every $x\in X$ equation \eqref{e6.1} has a unique generalized  solution and that the associated transition semigroup $T(t)$ defined by 
 $$(T(t)\varphi )(x) := \E[\varphi (X(t,x)], \quad \varphi\in C_b(X), \; t\geq0, $$ 
 possesses a unique invariant measure $\nu_R$,   see e.g. \cite[Ch. 4]{DP04}.  So, $T(t)$ may be extended to a contraction semigroup $T_p(t)$ to all spaces $L^p(X, \nu_R)$, $p\in [1, +\infty)$.  

For $\gamma=0$ the measure $\nu_R$ is an explicit weighted Gaussian measure, 
 $$\nu_R(dx) = \frac{1}{Z} e^{2U(x)}N_{0, Q}(dx)$$ 
where $N_{0, Q}$ is the Gaussian measure with mean $0$ and covariance  $Q= -A^{-1}/2$, the function $U$ is defined by 
$$U(x)= \left\{ \begin{array}{lll}  & \ds \int_0^{1} f(x )d\xi, & x\in L^{d}(0,1), 
\\
\\
 & -\infty , & x \notin L^{d}(0,1), \end{array}\right. 
$$
and $Z= \int_X e^{2U}dN_{0, Q}$. See \cite[Sect. 5]{DaLu14}. 
Since $U$, $ e^{2U}\in W^{1, p}(X,N_{0, Q})$ for every $p>1$ by  \cite[Sect. 5]{DaLu14}, $\nu_R$ is one of the measures considered in Section 6.  

For $\gamma >0$, $\nu_R$ is not
   explicit.

The following result  is proved in \cite[Th. 1.2]{DaDe15} for $\delta <1-\gamma$, in \cite[Th. 10]{DaDe15a} for $\delta= 1-\gamma$.

 \begin{Theorem}
  \label{t6.1}
Let  $\delta\in (0,1-\gamma]$, $p\in (1,\infty)$. Then  there exists  $C_p>0$ such that  for all $\varphi\in C^1_b(X )$   we have
\begin{equation}
\label{reazdiffh}
\bigg| \int _X \langle \nabla \varphi(x),h\rangle  \, \nu_R (dx)\bigg|  \le C_p\|\varphi\|_{L^p(X,\nu_R)} \,\|h\|_{H^{1+\delta +\gamma}(0,1)},\quad  h\in H^{ 1+\delta+ \gamma}(0,1)\cap H^1_0(0,1).
\end{equation}
\end{Theorem} 

Setting  $ h = (-A)^{-(1+\delta +\gamma)/2}k$ with $k\in X$, formula \eqref{reazdiffh} may be rewritten as
$$\bigg| \int _X \langle (-A)^{-(1+\delta +\gamma)/2}\nabla \varphi(x),k\rangle  \, \nu_R (dx)\bigg| \le C_p\|\varphi\|_{L^p(X,\nu_R)} \,\|k\| ,\quad  k\in X. $$
Therefore,  fixed any $\beta\in((1+\gamma)/2,1]$, Hypothesis \ref{h1'} is  fulfilled with $R=(-A)^{-\beta}$. With this choice of $R$, Hypothesis \ref{h1} too  is fulfilled, and  we can consider the  operators $M_p$ and their adjoint operators $M_p^*$ described in Sections 1, 2 for $p\in (1, +\infty)$. We do not know whether Hypothesis \ref{h3} holds. 

To define  surface measures on the level sets  of a function $g:X\mapsto \R$, we need that $g$ satisfies Hypothesis \ref{h2}. 
If $g: X\mapsto \R$ is a twice Fr\'echet differentiable function, the vector field $\Psi$ in formula \eqref{psi} is given by 
\begin{equation}
\label{psiReazDiff}
\Psi(x) = \frac{(-A)^{-\beta} \nabla g(x)}{\|(-A)^{-\beta} \nabla g(x)\|^2} = \frac{1}{\|(-A)^{-\beta} \nabla g(x)\|^2} \sum_{h=1}^{\infty} \alpha_h^{-\beta} \partial_{e_h} g (x) e_h, \quad x\in X.
\end{equation}

We present below two examples of smooth functions $g$ that satisfy Hypothesis \ref{h2}, namely such that $g\in W^{1,p}(X, \nu_R)$ and $\Psi\in D(M^*_p)$ for every $p>1$.

\subsubsection{Spherical surfaces}

 Let $g(x):=\|x\|^2$. Theorem 4.20 of \cite{DP04} and the H\"older inequality yield $g\in L^d(X, \nu_R)$ where $d$ is the degree of $f$.   The arguments of \cite{DP04} can be easily carried on to improve this result.

\begin{Lemma}
\label{Le:sommabilita'}
\begin{itemize}
\item[(i)] $\nu_R(L^q(0,1)) = 1$ for every $q\geq 2$;
\item[(ii)] $x\mapsto\|x\|^2 \in L^p(X, \nu_R)$ for every $p>1$.
\end{itemize} 
\end{Lemma}
\begin{proof}
 We follow the proof of Theorem 4.20 of \cite{DP04}, replacing $2d$ by $2m$ with $m\in \N$, and obtaining 
\begin{equation}
\label{summ}
\int_X \|x\|^{2m}_{L^{2m}(0,1)} \nu_R(dx) <\infty, \quad m\in \N. 
\end{equation}
Therefore, the function $x\mapsto \|x\| _{L^{2m}(0,1)}$ has finite values $\nu_R$-a.e., namely 
$\nu_R (L^{2m}(0,1)) =1$ for every $m\in \N$, which is statement (i). By the H\"older inequality, $\|x\|_X \leq   \|x\|_{L^{2m}(0,1)}$ for every $x\in L^{2m}(0,1)$, 
and statement (ii) follows. 
\end{proof}

Lemma \ref{Le:sommabilita'} yields that $g\in L^p(X, \nu)$ for every $p>1$. 

As we mentioned in the Introduction, the verification of Hypothesis \ref{h2} will be reduced to check that $ \|Mg(\cdot)\|^{-1}$ belongs to $L^p(X, \nu)$ for every $p>1$. In this case, $\|Mg(x)\|^{-1} =( 2\|(-A)^{-\beta}x\|)^{-1}$, and the $p$-summability of this function is not obvious.

To begin with, we prove that suitable smooth cylindrical functions belong to the domain of the infinitesimal generator $L$ of $T_2(t)$. This will be used to get estimates through the equality $\int_X L\varphi \,d\nu_R =0$, which holds for every $\varphi \in D(L)$.

 \begin{Lemma}
 \label{cilindriche}
For every $n\in \N$ and $\theta \in  C^2_b(\R^n)$ the function $\varphi (x):= \theta(\langle x, e_1\rangle, \ldots \langle x, e_n\rangle )$ belongs to the domain of the infinitesimal generator $L$ of $T_2(t)$, and 
\begin{equation}
\label{Kolm}
\begin{array}{lll}
L\varphi (x) & = &\ds \frac12\;\mbox{\rm Tr}\;[(-A)^{-\gamma} D^2\varphi  ]+\langle x ,A \nabla \varphi (x) \rangle + \langle f(x),\nabla \varphi(x) \rangle =
\\
\\
&= & \ds\frac12 \sum_{h=1}^n \alpha_h^{-\gamma} \frac{\partial^2 \theta}{\partial \xi_h^2}(x_1, \ldots x_n) - \sum_{h=1}^n \alpha_h x_h   \frac{\partial \theta}{\partial \xi_h}(x_1, \ldots x_n) +
\sum_{h=1}^n \langle f(x),  e_h\rangle \frac{\partial \theta}{\partial \xi_h}(x_1, \ldots x_n). 
\end{array}
\end{equation}
\end{Lemma}
\begin{proof}
By \cite[Thm. 4.23]{DP04},  $L$ is the closure of the operator $L_0: {\mathcal E}_A(X)\mapsto L^2(X, \nu_R)$ defined by $L_0\psi(x)=  \frac12\;\mbox{\rm Tr}\;[(-A)^{-\gamma}  D^2\psi  ]+\langle x ,A \nabla \psi (x) \rangle + \langle f(x),\nabla \psi(x) \rangle $ for $\psi \in {\mathcal E}_A(X)$. To prove that $\varphi\in D(L)$ it is sufficient to approach $\varphi$ by a sequence $(\psi_k)$ of elements of $ {\mathcal E}_A(X)$ in $L^2(X, \nu_R)$, such  that the sequence $L_0\psi_k$ converges in $L^2(X, \nu_R)$. 

By Lemma \ref{Le:Fejer} there exists a sequence of trigonometric polynomials $(\theta_k)$ such that $\theta_k$ and its first and second order derivatives converge pointwise to $\theta$ and to its first and second order derivatives,respectively,  and moreover $\|\theta_k\|_{C^2_b(\R^n)} \leq C$ independent of $k$. We set 
\begin{equation}
\label{psik}
\psi_k(x) = \theta_k (x_1, \ldots x_n), \quad k\in \N, \; x\in X. 
\end{equation}
Since $P_n(X)\subset D(A)$, $\psi_k\in  {\mathcal E}_A(X)$ for every $k\in \N$. By the Dominated Convergence Theorem, $\psi_k\to \varphi$ in $L^2(X, \nu_R)$ as $k\to \infty$. 
Moreover, $\partial_j \psi_k(x) = \partial \theta_k/\partial \xi_j(x_1, \ldots x_n)$, $\partial_{ij }\psi_k(x) = \partial ^2\theta_k/\partial \xi_{i}\partial \xi_j(x_1, \ldots x_n)$ if $i$, $j\leq n$, 
and $\partial_j \psi_k(x) = \partial_{ij }\psi_k(x)  =0$ otherwise. So, for every $x\in L^{2d}(0,1)$ (and hence, almost everywhere)
$$L_0  \psi_k(x) =  \sum_{h=1}^n \alpha_h^{-\gamma} \frac{\partial^2 \theta_k}{\partial \xi_h^2}(x_1, \ldots x_n) - \sum_{h=1}^n \alpha_h x_h   \frac{\partial \theta_k}{\partial \xi_h}(x_1, \ldots x_n) +
\sum_{h=1}^n \langle f(x),  e_h\rangle \frac{\partial \theta_k}{\partial \xi_h}(x_1, \ldots x_n). $$
Therefore, $L_0  \psi_k$ converges pointwise a.e. to the function in the right-hand side of \eqref{Kolm}. Since $f$ is a polynomial, by Statement (i) of Lemma \ref{Le:sommabilita'}  for every $h=1, \ldots, n$ the function $x\mapsto \langle f(x),  e_h\rangle$ belongs to $L^2(X, \nu_R)$, as well as the function $x\mapsto x_h$. Therefore, $|L_0  \psi_k(x)|\leq g(x)$ where $g$ is an $L^2$ function independent of $k$, and again by the Dominated Convergence Theorem the sequence $(L_0  \psi_k)$ converges to the function in the right-hand side of \eqref{Kolm} in $L^2(X, \nu_R)$. 
 \end{proof}

\begin{Proposition}
\label{stimadenominatore}
If $\gamma \leq 1/2$, $x\mapsto \|(-A)^{-\beta}x\|^{-1} \in L^p(X, \nu_R)$ for every $p>1$. 
\end{Proposition} 
\begin{proof}  
Recalling that the sequence $(\alpha_n)$ is increasing, for every $n\in \N$ we estimate
$$
\frac1{\|(-A)^{-\beta}x\|^2}\le \frac1{\|(-A)^{-\beta}P_nx\|^2}\le \frac{\alpha_n^{2\beta}}{\|P_nx\|^2}, 
$$  
where $P_n$ is the projection on span ${e_1, \ldots e_n}$ defined in \eqref{Pn}. So, it is enough to show that for every $k\in \N$ there exists  $n\in \N$ such that 
\begin{equation}
\label{e6.14}
x\mapsto \frac1{\|P_nx\|^2}\in L^{k+1}(X,\nu_R) .
\end{equation}
We shall show that \eqref{e6.14} holds for large enough $n$. 
To this aim we approach $1/\|P_nx\|^2$ by the smooth functions
$$
\varphi_\eps(x):=\frac1{(\eps +\|P_nx\|^2)^k},\quad x\in X, 
$$
that belong to the domain of the infinitesimal generator $L$ of the transition semigroup by  Lemma \ref{cilindriche}. For every $h$, $h_1$, $h_2\in X$ we have
$$
\langle \nabla \varphi_\eps(x),h\rangle=-\frac{2k\langle P_nx,P_nh\rangle }{(\eps+\|P_nx\|^2)^{k+1}}, 
$$
and
$$
 D^2\varphi_\eps(x)(h_1,h_2)=-2k\frac{\langle P_nh_1,P_nh_2\rangle}{(\eps+\|P_nx\|^2)^{k+1}}+4k(k+1)\frac{\langle P_nx,P_nh_1\rangle\,\langle P_nx,P_nh_2\rangle}{(\eps+\|P_nx\|^2)^{k+2}}.
$$
Therefore, 
$$\frac12\,\mbox{\rm Tr}\;[(-A)^{-\gamma} D^2\varphi_\eps(x)]=
- \frac{k\sum_{j=1}^n \alpha_j^{-\gamma}}{(\eps+\|P_nx\|^2)^{k+1}}+2k(k+1)\frac{\|(-A)^{-\gamma/2}P_nx\|^2}{(\eps+\|P_nx\|^2)^{k+2}}. 
$$
So, \eqref{Kolm} yields
\begin{equation}
\label{e6.15}
\begin{array}{lll}
\ds L\varphi_\eps(x)&=&\ds- \frac{k\sum_{j=1}^n \alpha_j^{-\gamma} }{(\eps+\|P_nx\|^2)^{k+1}}+2k(k+1)\frac{\|(-A)^{-\gamma/2}P_nx\|^2}{(\eps+\|P_nx\|^2)^{k+2}}\\
\\
&&\ds  -\frac{2k\langle AP_nx,x \rangle }{(\eps+\|P_nx\|^2)^{k+1}} -\frac{2k\langle P_nx,f(x)\rangle }{(\eps+\|P_nx\|^2)^{k+1}}.
\end{array} 
\end{equation}
Since  $\nu_R$ is invariant we have
$$
\int_X L\varphi_\eps(x)\,\nu_R(dx)=0, 
$$
and therefore
\begin{equation}
\label{e6.16}
\begin{array}{lll}
\ds k\sum_{j=1}^n \alpha_j^{-\gamma}\,\int_H\frac{1}{(\eps+\|P_nx\|^2)^{k+1}}\,\nu_R(dx)& =& \ds 2k\int_H\frac{\|(-A)^{1/2}P_nx\|^2}{(\eps+\|P_nx\|^2)^{k+1}}\,\nu_R(dx)\\
\\
& &\ds 
-2k\int_H \frac{ \langle P_nx,f(x)  \rangle}{(\eps+\|P_nx\|^2)^{k+1}}\,\nu_R(dx)\\
\\
&& \ds +2k(k+1)\int_H\frac{\| (-A)^{-\gamma/2} P_nx\|^2}{(\eps+\|P_nx\|^2)^{k+2}}\,\nu_R(dx)\\
\\
& =& :I_1+I_2+I_3.
\end{array} 
\end{equation}
Let us estimate $I_1$. Since $(\alpha_n)$ is an increasing sequence, 
$$
\|(-A)^{1/2}P_nx\|^2  \leq  \alpha_n \|P_nx\|^2 \leq \alpha_n(\eps+\|P_nx\|^2),$$
and using the H\"older and Young inequalities we obtain that for any $\delta>0$ there is $C_1(\delta,k,n)$ such that
\begin{equation}
\label{e6.18}
\begin{array}{lll}
 |I_1|&\le &\ds 2k\alpha_n\int_X\frac{1}{(\eps+\|P_nx\|^2)^{k}}\,\nu_R(dx)\le 2k\alpha_n\left(\int_X\frac{1}{(\eps+\|P_nx\|^2)^{k+1}}\,\nu_R(dx)\right)^{\frac{k}{k+1}}\\
\\
&\le&\ds C_1(\delta,k,n)+\delta\int_X\frac{1}{(\eps+\|P_nx\|^2)^{k+1}}\,\nu_R(dx).
\end{array} 
\end{equation}
Let us estimate $I_2$. Since
$$
| \langle P_nx,f(x)  \rangle|\le \|P_nf(x)\|\,\|P_nx\|\le
\|f(x)\|\,(\eps+\|P_nx\|^2)^{1/2},
$$
arguing as before and taking \eqref{summ} into account, we see that for any $\delta>0$ there is $C_2(\delta,k)$ such that 
\begin{equation}
\label{e6.19}
\begin{array}{lll}
  |I_2|&\le &\ds 2k \int_X\frac{\|P_nf(x)\|}{(\eps+\|P_nx\|^2)^{k+1/2}}\,\nu_R(dx)\\
  \\
  &\le&\ds 2k \left(\int_X\|f(x)\|^{2k+2}\,\nu_R(dx)   \right)^{\frac1{2k+2}}\;\left(\int_X\frac{1}{(\eps+\|P_nx\|^2)^{k+1}}\,\nu_R(dx)\right)^{\frac{2k+1}{2k+2}}\\
\\
&\le&\ds C_2(\delta,k)+\delta\int_X\frac{1}{(\eps+\|P_nx\|^2)^{k+1}}\,\nu_R(dx).
\end{array} 
\end{equation}
To estimate  $I_3$ we recall once again  that $(\alpha_n)$ is an increasing sequence, so that $\|(-A)^{-\gamma/2}P_nx\|^2 \leq \alpha_1^{-\gamma} 
\|P_nx\|^2 $. Then
\begin{equation}
\label{e6.20}
|I_3|\le \frac{2k(k+1)}{\alpha_1^{\gamma}} \int_X\frac{1}{(\eps+\|Px\|^2)^{k+1}}\,\nu_R(dx).
\end{equation}
Estimates  \eqref{e6.18}--\eqref{e6.20} yield
\begin{equation}
\label{e6.21}
k\sum_{j=1}^n \alpha_j^{-\gamma}\,\int_X\frac{\nu_R(dx)}{(\varepsilon+\|P_nx\|^2)^{k+1}}
\le C_1(\delta,k,n)+C_2(\delta,k)+\bigg(\frac{2k(k+1)}{\alpha_1^{\gamma}}+2\delta \bigg)\int_X \frac{\nu_R(dx)}{(\varepsilon+\|P_nx\|^2)^{k+1}}.
\end{equation}
Since $\gamma \leq 1/2$, the series $s_n= \sum_{j=1}^n \alpha_j^{-\gamma}$ is divergent (recall that $\alpha_j= \pi^2j^2$). 
Now we choose $n$ and $\delta$ such  that
$$
k\sum_{j=1}^n \alpha_j^{-\gamma} > \frac{2k(k+1)}{\alpha_1^{\gamma}}+2\delta 
$$
and we conclude that there exists $M>0$, independent of $\eps$,  such that
\begin{equation}
\label{e6.22}
\int_X\frac{1}{(\eps+\|P_nx\|^2)^{k+1}}\,\nu_R(dx)\le M.
\end{equation}
Letting $\eps\to 0$  concludes the proof.
\end{proof}

With the aid of Lemma \ref{Le:sommabilita'} and Proposition \ref{stimadenominatore} we prove the main result of this section. 

\begin{Proposition}
\label{h3RD}
If $0\leq \gamma \leq 1/2$, the function $g(x) = \|x\|^2$ satisfies Hypothesis \ref{h2}. 
\end{Proposition}
\begin{proof}
$g$ is smooth and it belongs to $L^p(X, \nu_R)$ for every $p>1$ by Lemma  \ref{Le:sommabilita'}(ii). Moreover, $(-A)^{-\beta}\nabla g(x) = 2 (-A)^{-\beta}x$ for every $x\in X$, and since 
$\|(-A)^{-\beta}x\| \leq \pi^{-2\beta}\|x\|$, still by Lemma  \ref{Le:sommabilita'}(ii) $x\mapsto \|(-A)^{-\beta}\nabla g(x)\|$ $\in$ $ L^p(X, \nu_R)$ for every $p>1$. By Lemma \ref{C^1crescita}, $g\in W^{1,p}(X, \nu_R)$ for every $p>1$. 

It remains to prove that the vector field $\Psi  $ in formula \eqref{psi} belongs to $D(M^*_p)$ for every $p>1$. It is given by (see \eqref{psiReazDiff}) 
\begin{equation}
\label{Psi}
\Psi(x) = \frac{(-A)^{-\beta}x}{2\|(-A)^{-\beta}x\|^2}  
= \lim_{n\to \infty} \Psi_n(x), 
\end{equation}
where  
$$ \Psi_n(x) = \sum_{h=1}^n \frac{\alpha_h^{-\beta}x_h}{2\|(-A)^{-\beta}x\|^2}e_h =:  \sum_{h=1}^n \psi_h(x) e_h .  $$
Approaching every $\psi_h$ by the $C^1_b$  functions $\psi_{h,\eps}(x) := \alpha_h^{-\beta} x_h/ 2(\|(-A)^{-\beta}x\|^2 +\eps)$ and using  Proposition \ref{stimadenominatore},  one sees easily that $\psi_h$ belongs to $W^{1,p}(X, \nu_R)$ for every $p>1$, and 
 $$\langle M \psi_h(x),  e_h \rangle  = \alpha_h^{-2\beta}/2\|(-A)^{-\beta}x\|^2 - \alpha_h^{-4\beta} x_h^2/\|(-A)^{-\beta}x\|^2. $$
 By Lemma \ref{Rem:divergenza},  
$\Psi_n$ 
 belongs to $D(M^*_p)$ for every $p>1$, and  by \eqref{formulaM^*_p} we get  
\begin{equation}
\label{M^*Psi_n}
M^*_p \Psi_n(x) = - \sum_{h=1}^n \frac{\alpha_h^{-2\beta} }{2\|(-A)^{-\beta}x\|^2}  + \sum_{h=1}^n \frac{\alpha_h^{-4\beta} x_h^2}{2\|(-A)^{-\beta}x\|^4}
+ \sum_{h=1}^n \frac{\alpha_h^{-\beta} x_hv_{e_h}(x)}{2\|(-A)^{-\beta}x\|^2}. 
\end{equation}
Recalling that the series $\sum_{h=1}^n \alpha_h^{-2\beta}$ converges, that $\|v_{e_h}\|_{L^{p'}(X, \nu_R)}$ is bounded by a constant independent of $h$, and using Lemma  \ref{Le:sommabilita'} and Proposition \ref{stimadenominatore}, we easily deduce that $(M^*_p \Psi_n)$ converges in $L^{p'}(X, \nu_R)$, for every $p>1$. Therefore, $\Psi\in D(M^*_p)$ for every $p>1$, and Hypothesis \ref{h2} is satisfied. \end{proof}

 \subsubsection{Hyperplanes} 
 
Let $g(x)=\langle x,b\rangle$,  where $b\in X\setminus \{0\}$. $g$ is smooth, it has constant gradient,  and $Mg(x) = (-A)^{-\beta}b$ (constant).  Therefore, $g$ belongs to all spaces $W^{1,p}(X, \nu_R)$, for $p>1$, by Lemmas \ref{Le:sommabilita'} and \ref{C^1crescita}. The vector field $\Psi = Mg/\|Mg\|^2$ is also constant and it is given by 
 $$\Psi(x) =\frac{(-A)^{-\beta}b}{\|(-A)^{-\beta}b\|^2}, \quad x\in X.$$
 Since Hypothesis \ref{h1'} is satisfied,  $ \Psi$ belongs to $D(M_p^*)$ for every $p>1$, and we have
 $$M_p^*  \Psi =  \frac{v_{(-A)^{-\beta}b}}{\|(-A)^{-\beta}b\|^2}. $$
  Therefore, $g$ satisfies Hypothesis  \ref{h2}.


\subsection{Burgers equation}
\label{Burgers}

We are concerned with the   stochastic differential  equation \eqref{e6.1}
with $\gamma=0$ and 
$$f(x)= 2x x',\quad x\in H^1_0(0,1), $$
where the prime denotes the weak derivative.   It is well known that for every $x\in X$,  equation \eqref{e6.1} has a unique mild solution,  and that the associated transition semigroup $P(t)$, defined on $C_b(X)$ by 
$$P(t) \varphi (x):= \E [\varphi (X(t, x)], \quad t\geq 0, \; x\in X, $$
possesses a unique invariant measure $\nu_B$, see e.g. \cite[Thm. 14.4.4]{DPZ97}. So, $P(t)$ may be extended to a strongly continuous semigroup $P_p(t)$ in $L^p(X, \nu_B)$, for every $p\geq 1$. 
  
A result analogous to Theorem \ref{t6.1} was proved in \cite[Theorem 2]{DaDe16}.

\begin{Theorem}
\label{t7.1}
For any $p>1$, $\delta>0$, there exists $C>0$ such that for all  $\varphi\in  C^1_b(X)$ and all $h\in H^{1+\delta}(0,1)\cap H^1_0(0,1)$, we have
\begin{equation}
\label{e2d}
\left|\int _X \langle D\varphi(x),h\rangle \, \nu_B(dx)\right| \le C\|\varphi\|_{L^p(X,\nu_B)} \,\|h\|_{H^{1+\delta}(0,1)}. 
\end{equation}
\end {Theorem}

As in Section \ref{Reaction--Diffusion},   it follows that Hypotheses \ref{h1} and \ref{h1'}  are fulfilled with $R=A^{-\beta}$ for  all $\beta\in (1/2, 1)$.  Also in this case, we do not know whether Hypothesis \ref{h2} holds. And also in this case we are going to show that our theory fits to spherical surfaces and to hyperplanes. The proofs are similar to the proofs in Section  \ref{Reaction--Diffusion} and we only sketch them. 
  
 \medskip

Let $g(x) := \|x\|^2$. It was proved in \cite[Prop. 2.3]{DaDe07} that 
\begin{equation}
\label{momentsBurgers}
\int_X \|x\|^k_{L^q(0,1)}\, \nu_B (dx) < +\infty, \quad k\in \N, \; q\geq 2. 
\end{equation}  
It follows that  $ \nu_B (L^q(0,1)) = 1$ for every $q\geq 2$, and that $g\in L^p(X, \nu_B)$ for every $p>1$. To prove that $g$ satisfies Hypothesis \ref{h2}, we argue as in  Proposition \ref{h3RD}. 
First, we remark that  
$g\in W^{1,p}(X, \nu_B)$ for every $p>1$, by  \eqref{momentsBurgers} and Lemma \ref{C^1crescita}. 
Second, the vector field $\Psi = Mg/\|Mg\|^2$ is still given by formula \eqref{Psi}. Proving that it belongs to $D(M^*_p)$ for every $p>1$ amounts to show that $x\mapsto \|(-A)^{-\beta}x\|^{-2}$ belongs to $L^p(X, \nu_B)$ for every $p>1$. This can be proved as in the case of reaction-diffusion equations, with the aid of the following lemma.

 \begin{Lemma}
 \label{cilindricheBurgers}
For every $n\in \N$ and $\theta \in  C^2_b(\R^n)$ the function $\varphi (x):= \theta(x_1, \ldots x_n)$ belongs to the domain of the infinitesimal generator $N$ of $P_2(t)$, and 
\begin{equation}
\label{KolmB}
\begin{array}{lll}
N\varphi (x) & = & \frac12\;\mbox{\rm Tr}\;[ D^2\varphi  ]+\langle x ,A \nabla \varphi (x) \rangle + \langle x^2, (\nabla \varphi(x))' \rangle =
\\
\\
&= & \ds  \frac12 \sum_{h=1}^n \frac{\partial^2 \theta}{\partial \xi_h^2}(x_1, \ldots x_n) - \sum_{h=1}^n \alpha_h x_h   \frac{\partial \theta}{\partial \xi_h}(x_1, \ldots x_n) -
\sum_{h=1}^n  \frac{\partial \theta}{\partial \xi_h}(x_1, \ldots x_n) \langle x^2,  e_h' \rangle . 
\end{array}
\end{equation}
\end{Lemma}
\begin{proof}
By \cite[\S 4.1]{DaDe07},  $N$ is the closure of the operator $N_0: {\mathcal E}_A(X)\mapsto L^2(X, \nu_R)$ defined by 
$N_0\psi(x)=  \frac12\;\mbox{\rm Tr}\;[ D^2\psi  ]+\langle x ,A \nabla \psi (x) \rangle - \langle x^2,(\nabla \psi(x))' \rangle $ for $\psi \in {\mathcal E}_A(X)$. 
In fact, $N_0 \psi (x)$ is formally defined by 
$$N_0\psi(x)=  \frac12\;\mbox{\rm Tr}\;[ D^2\psi  ]+\langle x ,A \nabla \psi (x) \rangle + \langle 2x x', \nabla \psi(x) \rangle,  $$
which is meaningful for $x\in H^1(0,1)$. However, we do not know whether $\nu_B(H^1(0,1)) =1$ so that the scalar product $ \langle 2x x', \nabla \psi(x) \rangle$ has to be rewritten in the more convenient way
$\langle x^2,(\nabla \psi(x))' \rangle $, obtained just integrating by parts.

As in Lemma \ref{cilindriche}, we approach  $\varphi$ by a sequence $(\psi_k)$ of elements of $ {\mathcal E}_A(X)$ in $L^2(X, \nu_B)$, such  that the sequence $L_0\psi_k$ converges in $L^2(X, \nu_B)$. 
 $(\psi_k)$ is the   sequence defined in \eqref{psik}, and it converges to $\psi$ in $L^2(X, \nu_B)$ by the Dominated Convergence Theorem. Moreover, 
$$N_0\psi_k(x)= \frac{1}{2}  \sum_{h=1}^n \frac{\partial^2 \theta_k}{\partial \xi_h^2}(x_1, \ldots x_n) - \sum_{h=1}^n \alpha_h x_h   \frac{\partial \theta_k}{\partial \xi_h}(x_1, \ldots x_n)
 - \sum_{h=1}^n  \frac{\partial \theta_k}{\partial \xi_h}(x_1, \ldots x_n) \langle x^2,  e_h' \rangle $$
which converges pointwise  to the function in the right-hand side of \eqref{KolmB}. Moreover, $|N_0\psi_k(x)| \leq C\|\theta \|_{C^2_b(\R^n)} ( 1 + \|x\| + \|x\|^2) $ which is in $L^2(X, \nu_B)$ by \eqref{momentsBurgers}, and again by the Dominated Convergence Theorem the sequence $(N_0  \psi_k)$ converges to the function in the right-hand side of \eqref{KolmB} in $L^2(X, \nu_B)$. 
 \end{proof}

\begin{Proposition}
\label{l7.4}
$$
x\mapsto \frac1{\|(-A)^{-\beta}x\|^2}\in L^{k+1}(X,\nu_B),\quad\forall\,k\in\N.
$$
\end{Proposition}\begin{proof}
We follow the proof of Proposition \ref{stimadenominatore}. For every  $n\in \N$   we estimate
$$
\frac1{\|(-A)^{-\beta}x\|^2}\le \frac1{\|(-A)^{-\beta}P_nx\|^2}\le \frac{\alpha^\beta_n}{\|P_nx\|^2}.
$$  
Then it is enough to show that for each $k\in \N$ there is $n\in \N$ such that 
\begin{equation}
\label{e7.28}
\frac1{\|P_nx\|^2}\in L^{k+1}(X,\nu_B),
\end{equation}
and to this aim we approach $1/\|P_nx\|^{2(k+1)}$ by the functions 
$$
\varphi_\varepsilon(x)=\frac1{(\varepsilon+\|P_nx\|^2)^{k+1}}, 
$$
that belong to $D(N)$ by Lemma \ref{cilindricheBurgers}. Formula \eqref{KolmB} (recall that now $\gamma=0$) yields
\begin{equation}
\label{e7.29}
\begin{array}{lll}
\ds N \varphi_\varepsilon (x)&=&\ds- \frac{kn}{(\varepsilon+\|P_nx\|^2)^{k+1}}+2k(k+1)\frac{\|P_nx\|^2}{(\varepsilon+\|P_nx\|^2)^{k+2}}\\
\\
&&\ds  -\frac{2k\langle AP_nx,x\rangle}{(\varepsilon+\|Px\|^2)^{k+1}} +  \frac{2k\langle (P_nx)', x^2\rangle}{(\varepsilon+\|Px\|^2)^{k+1}}.
\end{array} 
\end{equation}
Since
$$
\int_X N \varphi_\varepsilon(x)\,\nu_B(dx)=0
$$
by the invariance of $\nu_B$, we find
\begin{equation}
\label{e7.30}
\begin{array}{l}
\ds kn \,\int_X\frac{1}{(\varepsilon+\|P_nx\|^2)^{k+1}}\,\nu_B(dx)=2k\int_X\frac{\|(-A)^{1/2}P_nx\|^2}{(\varepsilon+\|P_nx\|^2)^{k+1}}\,\nu_B(dx)\\
\\
\ds + 2k\int_X \frac{ \langle (P_nx)',x^2 \rangle}{(\varepsilon+\|P_nx\|^2)^{k+1}}\,\nu_B(dx) +2k(k+1)\int_X\frac{\| Px\|^2}{(\varepsilon+\|Px\|^2)^{k+2}}\,\nu_B(dx)\\
\\
=:I_1+I_2+I_3.
\end{array} 
\end{equation}
Estimates of $I_1$ and $I_3$ are identical to the corresponding ones in the proof of Proposition \ref{stimadenominatore} with  $\gamma=0$; to estimate  $I_2$ we need different arguments. 
We have
$$\langle x^2,  (P_nx)'\rangle = \int_0^1 (x(\xi))^2  \sum_{h=1}^n \langle x, e_h\rangle e_h'(\xi) \,d\xi $$
so that 
$$\begin{array}{lll}
|\langle x^2,  (P_nx)'\rangle | & \leq & \ds   \bigg(\int_0^1 x^4d\xi\bigg)^{1/2}  \bigg(\int_0^1 \bigg(\sum_{h=1}^n \langle x, e_h\rangle e_h'(\xi)\bigg)^2d\xi\bigg)^{1/2} \leq \|x\|_{L^4(0,1)}^2 C_n \|P_nx\|
\\
\\
&\leq & \|x\|_{L^4(0,1)}^2 C_n (\varepsilon+\|P_nx\|^2)^{1/2}, 
\end{array}$$
and therefore
\begin{equation}
 \label{e8.2}
\begin{array}{lll}
  |I_2|&\le &\ds 2kC_n \int_X\frac{\|x\|_{L^4(0,1)}^2}{(\varepsilon+\|Px\|^2)^{k+1/2}}\,\nu_B(dx)\\
  \\
  &\le&\ds 2kC_n \left(\int_X\|x\|_{L^4(0,1)}^{2k+2}\,\nu_B(dx)   \right)^{\frac1{2k+2}}\;\left(\int_X\frac{1}{(\varepsilon+\|Px\|^2)^{k+1}}\,\nu_B(dx)\right)^{\frac{2k+1}{2k+2}}.
  \end{array} 
\end{equation}
Since $\int_X\|x\|_{L^4(0,1)}^{2k+2}\,\nu_B(dx) <\infty $ by \eqref{momentsBurgers},   there exists a constant $C(k,n)>0$ such that
\begin{equation}
 \label{e8.3}
\begin{array}{lll}
  |I_2|&\le & \ds C(k,n) \left(\int_X\frac{1}{(\varepsilon+\|P_nx\|^2)^{k+1}}\,\nu_B(dx)\right)^{\frac{2k+1}{2k+2}}\\
\\
&\le&\ds C(k,n,\delta)+\delta\int_X\frac{1}{(\varepsilon+\|P_nx\|^2)^{k+1}}\,\nu_B(dx),
\end{array} 
\end{equation}
for any $\delta>0$ and a suitable $ C(k,n,\delta)>0$, by Young's inequality. The conclusion follows now as in the proof of Proposition \ref{stimadenominatore} .
 \end{proof}

The procedure of Subsection 8.1.2 works as well in this case, without any modification. Therefore, for every $b\in X\setminus \{0\}$ the function $g(x) :=\langle x, b\rangle$ satisfies Hypothesis 
  \ref{h2}. 
 
\section{Final remarks and bibliographical notes}

{\em 1. Sobolev spaces.} The theory of Sobolev spaces for differentiable measures is well developed only in the Gaussian case. See \cite{Boga} for Gaussian measures in general locally convex spaces, 
\cite{DPZ} for Gaussian measures in Hilbert spaces. Basic results for general differentiable measures are in \cite[Ch. 2]{Boga2}. 

We did not consider the space $W^{1,1}(X, \nu)$, which is a very special case (even for Gaussian measures) and would deserve a specific treatment. Together with  $W^{1,1}(X, \nu)$,  spaces of  BV functions 
are still to be thoroughly investigated.  Some initial results are in \cite{RZZ}. The case of weighted Gaussian measures in Hilbert spaces was considered in \cite{AmDaGoPa12}. 

Sobolev spaces of functions defined in (smooth) domains rather than in the whole $X$ are even more puzzling. Even in the case of Gaussian measures the theory is far from being complete. A major difficulty comes from the lack of a bounded extension operator from $W^{1,p}(\Omega, \nu)$ to $W^{1,p}(X, \nu)$: see \cite{Bogaetal} for a counterexample. If $X$ is a separable infinite dimensional Hilbert space and $\nu$ is a nondegenerate Gaussian measure in $X$, the existence of a bounded extension operator from $W^{1,2}(B(0,1), \nu)$ to $W^{1,2}(X, \nu)$ is still an open question. 

   \vspace{3mm}
         
{\em 2. Surface measures.}  For a detailed account on the existing literature on surface measures in infinite dimension, we refer to the survey paper \cite{BogaReview}. 

 Hypothesis \ref{h2} on the defining function $g$  is our main assumption. It could be replaced by $Mg/\|Mg\|^2\in D(M^*_{\overline{p}})$ for {\em some}  $\overline{p}$, but this would lead to restrictions on the validity of several results. For instance, 
in Lemma \ref{l4} and in all its consequences  we should take $\varphi\in L^p(X, \nu)$ only  with $p\geq \overline{p}'$. 

Checking Hypothesis \ref{h2}
 in specific  examples  is  reduced to some regularity/summability assumptions on $g$, plus summability of $\|Mg\|^{-p}$ for every $p$. While the regularity and summability properties of $Mg$ can be considered standard conditions and can be checked in standard ways, to prove that $\|Mg\|^{-p}$ belongs to $L^1(X, \nu)$ is more difficult. To overcome this difficulty, we could replace the function $F_{\varphi}$ used throughout the paper by 
$$\widetilde{F}_{\varphi}(r) = \int_{\{g<r\}} \varphi(x) \|Mg(x)\|^2 \nu(dx), \quad r\in \R, $$
and replace Hypothesis  \ref{h2} by $Mg\in D(M^*_p)$ for every $p>1$, as suggested in \cite{BogaReview}. Then, the procedure of Lemma \ref{l4} yields that the measure $(\varphi \|Mg\|\nu)\circ g^{-1}$ is absolutely continuous with respect to the Lebesgue measure, with density
$$\widetilde{q}_{\varphi}(r) = \int_{\{g<r\}} (\langle M_p\varphi, Mg\rangle - \varphi M^*(Mg))d\nu, \quad r\in \R, $$
and the procedure of Theorem \ref{costruzione} gives a Borel measure $\widetilde{\sigma}_r^g$
such that $\widetilde{F}_\varphi '( r)  =\int_{X} \varphi(x)\,\widetilde{\sigma}_r^g(dx)$, 
for every $\varphi\in C_b(X)$. However, as the measures $\sigma^g_r$, these measures depend explicitly on $g$, and have not any intrinsic geometric or analytic meaning. The geometrically meaningful
 measure is what we called $\rho_r$, see Section 5, and to obtain it the assumption $\|Mg\|^{-p}\in L^1(X, \nu)$ for some $p$ seems to be unavoidable.

 \section{Acknowledgements}
 
 We thank V. Bogachev for enlightening discussions, and for letting us know  a preliminary version of      \cite{BogaReview}.    Our work was partially supported by the research project 
 PRIN 2010MXMAJR ``Evolution differential problems: deterministic and stochastic approaches and their interactions".

\end{document}